\documentclass[11pt]{amsart}
\usepackage[margin = 1in]{geometry}
\usepackage{amsmath, amssymb, amsthm}
\usepackage{enumerate}
\usepackage{hyperref}
\usepackage{url}
\usepackage{comment}

\usepackage{color}

\usepackage[usenames,dvipsnames]{xcolor}

\newcommand{\F}[0]{\mathbb{F}}
\newcommand{\R}[0]{\mathbb{R}}
\newcommand{\gl}[0]{\operatorname{GL}}
\newcommand{\tr}[0]{\operatorname{tr}}
\newcommand{\wt}[1]{\widetilde{#1}}
\newcommand{\frob}[0]{\operatorname{Frob}}
\newcommand{\card}[0]{\#}
\newcommand{\pmat}[4]{\begin{pmatrix}#1 & #2 \\ #3 & #4\end{pmatrix}}
\newcommand{\smat}[4]{\left(\begin{smallmatrix} #1 & #2 \\ #3 & #4\end{smallmatrix}\right)}

\newcommand{\Ov}{\mathcal{O}_{v}}
\newcommand{\Ok}{\mathcal{O}_{K}}
\newcommand{\Q}{\mathbb{Q}}
\newcommand{\Z}{\mathbb{Z}}
\newcommand{\calO}{\mathcal{O}}
\newcommand{\frakp}{\mathfrak{p}}
\newcommand{\frakq}{\mathfrak{q}}
\newcommand{\frakm}{\mathfrak{m}}
\newcommand{\nm}{\operatorname{N}_{K/\Q}}

\newcommand{\zz}{\mathbb{Z}}

\newcommand{\adelezz}{\widehat{\zz}}

\newcommand{\Quo}{\operatorname{Quo}}

\newcommand{\Gal}{\operatorname{Gal}}
\newcommand{\Mat}{\operatorname{Mat}}
\newcommand{\Zb}{\mathbb Z}
\newcommand{\cl}{\overline}
\newcommand{\mf}{\mathfrak}
\newcommand{\gk}{G_{K}}
\newcommand{\Zl}{\Zb_\ell}
\newcommand{\Qb}{\mathbb{Q}}
\newcommand{\Fl}{\mathbb{F}_\ell}

\newcommand{\Kcyc}{K^{\cyc}}
\newcommand{\calP}{\mathcal{P}}

\newcommand{\Cc}{\mathcal{C}}
\newcommand{\Dc}{\mathcal{D}}
\newcommand{\raym}{\mathcal{C}^{\frakm}_K}
\newcommand{\frakr}{\mathfrak{r}}

\DeclareMathOperator{\GL}{GL}
\DeclareMathOperator{\SL}{SL}
\DeclareMathOperator{\Frob}{Frob}
\DeclareMathOperator{\ab}{ab}
\DeclareMathOperator{\cyc}{cyc}
\DeclareMathOperator{\N}{N}

\newtheorem{thm}{Theorem}[section]
\newtheorem{lemma}[thm]{Lemma}

\newtheorem{prop}[thm]{Proposition}
\newtheorem{cor}[thm]{Corollary}
\newtheorem{ex}[thm]{Example}
\newtheorem{crit}[thm]{Criterion}

\theoremstyle{remark}
\newtheorem*{rem}{Remark}

\makeatletter
\def\imod#1{\allowbreak\mkern5mu({\operator@font mod}\,\,#1)}
\makeatother

\title{Elliptic Curves with Full 2-Torsion and Maximal Adelic Galois Representations}
\author{David Corwin, Tony Feng, Zane Kun Li, Sarah Trebat-Leder}

\address{Department of Mathematics, Princeton University, Princeton, New Jersey 08544}
\email{dcorwin@math.princeton.edu}

\address{479 Quincy Mail Center, 58 Plympton Street, Cambridge, MA 02138}
\email{tfeng@college.harvard.edu}

\address{Department of Mathematics, Princeton University, Princeton, New Jersey 08544}
\email{zkli@math.princeton.edu}

\address{Department of Mathematics, Princeton University, Princeton, New Jersey 08544}
\email{strebat@math.princeton.edu}

\begin{document}

\begin{abstract}
In 1972, Serre showed that the adelic Galois representation associated to a non-CM elliptic curve over a number field has open image in $\GL_2(\adelezz)$.
In \cite{greicius}, Greicius develops necessary and sufficient criteria for determining when this representation is actually surjective and exhibits such
an example. However, verifying these criteria turns out to be difficult in practice; Greicius describes tests for them that apply only to semistable
elliptic curves over a specific class of cubic number fields. In this paper, we extend Greicius' methods in several directions.  First, we consider
the analogous problem for elliptic curves with full 2-torsion. Following Greicius, we obtain necessary and sufficient conditions for the associated
adelic representation to be maximal and also develop a battery of computationally effective tests that can be used to verify these conditions.
We are able to use our tests to construct an infinite family of curves over $\Q(\alpha)$ with maximal image, where $\alpha$ is the real root of
$x^3 + x + 1$. Next, we extend Greicius' tests to more general settings, such as non-semistable elliptic curves over arbitrary cubic number fields.
Finally, we give a general discussion concerning such problems for arbitrary torsion subgroups.
\end{abstract}

\maketitle

\section{Introduction and Statement of Results}
Let $E$ be an elliptic curve over a number field $K$. For $m$ a positive integer, let $E[m]$ denote the group of $m$-torsion points of $E$ over $\overline{K}$.
It is well known that $E[m] \simeq \Z/m\Z \times \Z / m\Z$, and since $[m]$ is defined over $K$ the Galois group $G_K := \Gal(\overline{K}/K)$
acts on $E[m]$. This can be phrased as a Galois representation
\[
\rho_{E, m}: G_K \rightarrow \GL_2(\Z/m\Z).
\]
We call this the ``mod $m$'' Galois representation.
Now fix a prime $\ell$. Taking the inverse limit over powers of $\ell$, we obtain the \emph{Tate module} of $E$,
\[
T_{\ell}(E) := \varprojlim  \Z/\ell^n \Z
\]
Correspondingly, we have the ``$\ell$-adic'' Galois representation
\[
\rho_{E, \ell^{\infty}} \colon G_K \rightarrow \GL_2(\Z_{\ell}).
\]
Taking the product of these over all primes, we finally obtain the ``adelic'' Galois representation
\[
\rho_E \colon G_K  \rightarrow \GL_2(\adelezz).
\]

In \cite{serre}, Serre famously proved that for elliptic curves \emph{without} complex multiplication, the image of $\rho_E$ is always open in $\GL_2(\adelezz)$. In particular, this implies that $\rho_{E, \ell^{\infty}}$ is surjective for all sufficiently large $\ell$. He also showed that for an elliptic curve defined over $\mathbb{Q}$, $\rho_{E}$ is \emph{never} surjective.

Naturally, this raised the question of whether $\rho_E$ could be surjective for elliptic curves defined over number fields other than $\Qb$. In \cite{greicius}, Greicius proves necessary and sufficient abstract criteria for $\rho_E(G_K)$ to be the full group $\GL_2(\adelezz)$.  He then develops numerical tests applying to semistable elliptic curves over the field $\Q(\alpha)$, where $\alpha$ is the real root of $x^3 + x + 1 = 0$, which can be effectively used to verify these criteria. Through the tests he is able to give the first explicit example of an elliptic curve over a number field with surjective adelic Galois representation.  In doing so, the main challenge is to show that the $\ell$-adic representation is surjective for all $\ell$.


Despite the difficulty of finding explicit examples of surjective adelic representations, it turns out that much is known on average. Zywina, building on work of Duke and Jones in \cite{duke} and \cite{jones}, proved that almost all (in the sense of density) elliptic curves have surjective adelic
Galois representation (\cite{zyw1, zyw2}). In fact, his work shows that for any rational family of elliptic curves,
even when there are obstructions to the surjectivity of $\rho_E$ such as the presence of torsion defined over $K$,
the generic member has maximal Galois representation. However, his methods are ineffective in that they do not
explicitly produce any such elliptic curves.

In this paper, we first consider the question of determining the image of Galois for explicit elliptic curves under certain hypotheses
of torsion. More specifically, let $E/K$ be an elliptic curve with
\emph{full} $2$-torsion over $K$, so that $E[2](K) \cong \Z/2\Z \times \Z/2\Z$.
Then the action of $G_{K}$ on $E[2]$ is trivial, so the image of the $2$-adic Galois representation
\[
\rho_{E,2^{\infty}}\colon G_{K} \rightarrow \GL_2(\mathbb{Z}_{2})
\]
lies in the kernel of the reduction mod $2$ map.
In this case the adelic Galois representation clearly cannot be surjective. Writing $V_k(\ell)= I + \ell^k \Mat_{2 \times
2}(\mathbb{Z}_{\ell})$, the remarks above imply that $\rho_E$ factors through the map
\[
\rho_{E}\colon G_{K} \rightarrow V_1(2) \times \prod_{\ell > 2} \GL_2(\mathbb{Z}_{\ell}).
\]
Following Greicius, we first prove necessary and sufficient abstract criteria, analogous to Theorem 3.1
in \cite{greicius}, for the adelic Galois representation to be maximal
given that the elliptic curve has full two-torsion over its field of definition, $K$.

\begin{thm}\label{conditions}
Let $K$ be a number field. Then $\rho_E$ surjects onto
$V_1(2)  \times \prod_{\ell > 2} \GL_2(\mathbb{Z}_{\ell})$ if and only if the following three
conditions are satisfied:
\begin{enumerate}
\item $\rho_{E, \ell^{\infty}}$ is surjective for each $\ell \geq 3$ and $\rho_{E,2^{\infty}}$ surjects onto $V_1(2)$,
\item\label{intersection} $K(E[4]) \cap K^{cyc} = K(i)$,
\item $K \cap \mathbb{Q}^{\cyc} = \mathbb{Q}$.
\end{enumerate}
\end{thm}
\begin{rem}
Note that the third condition cannot be satisfied when $K$ is an abelian extension of $\mathbb{Q}$, so the simplest number field $K$ where this can be satisfied is a non-Galois cubic extensions.
\end{rem}
This theorem is not hard to prove; it is more difficult to test algorithmically if the first two conditions above are satisfied.
In the case where $K$ is the number field that Greicius considers, the first condition can be checked using his methods for
most $\ell$, but the case of the 2-adic representation is significantly different, as there is no $\GL_2(\Z/2\Z)$-module
structure on the image. In Section \ref{2adicsurj}, we develop new methods for handling this case. The second condition is
also unique to the two-torsion case, and requires some work to obtain a feasible computation.  Using the tests
developed by Greicius, as well as the ones mentioned above, we implemented a computer program to check maximality of adelic
representations. Simply iterating through coefficients, we found the following curve which has maximal adelic image.

\begin{ex}\label{single_example}
Let $K = \mathbb{Q}(\alpha)$, where $\alpha$ is the real root of $x^3+x+1$. Let
\[
E/K\colon y^{2} = x(x - (2\alpha^{2} + 7\alpha + 19))(x - (18\alpha^{2} + 7\alpha + 3)).
\]
Then $\rho_{E}$ surjects onto $V_1(2)  \times \prod_{\ell > 2} \GL_2(\mathbb{Z}_{\ell})$.
\end{ex}

Not only can our tests be used to check that specific elliptic curves have maximal adelic Galois representation, they can also be used to explicitly construct infinite families of such curves. This is done in Section \ref{infinite}, and we obtain the following result.

\begin{thm}\label{infinite}
Let $K = \mathbb{Q}(\alpha)$, with $\alpha$ as above. Let $M \colon= 2^{2}\cdot 3 \cdot 5 \cdot 7 \cdot 11 \cdot 13 \cdot 17 \cdot 29 \cdot 31 \cdot 47 \cdot 787 \cdot 827$ and
$$E_{b,c}/K \colon y^2 = x(x-(\alpha^2 + b\alpha + c))(x-16(\alpha^2+\alpha+1)).$$
If $b, c \in \Z$ satisfy
\begin{align*}
b &\equiv 17 \cdot 37 \cdot 257 \cdot 509^{2} \cdot 787 \cdot 827\imod{M}\\
c &\equiv 2^{4} \cdot 17 \cdot 787 \cdot 827 \cdot 4657 \cdot 15649 \imod{M},
\end{align*}
then the adelic representation $\rho_{E_{b, c}}$ surjects onto $V_1(2)  \times \prod_{\ell > 2} \GL_2(\mathbb{Z}_{\ell})$.
\end{thm}

As mentioned, Greicius's  tests depend upon special properties of the $E$ and $K$: they apply only to semistable elliptic curves with non-integral $j$-invariant, and also only to cubic fields $K$ with trivial narrow class group and possessing a totally positive unit $u$ such that $u-1$ is also a unit. More precisely, Greicius uses these conditions to prove that all primes $\ell$ failing condition 1 of Theorem \ref{conditions}, outside an easily computable finite set, satisfy a certain divisibility condition relating to the reduction of $E$ at primes of good reduction. In Section \ref{nonsemi}, we prove analogous results applying to all elliptic curves with non-integral $j$-invariant over arbitrary cubic number fields.

Using them, we obtain the following examples.  First, we give an example of a non-Galois cubic field $K$ and an elliptic curve $E/K$ with full $2$-torsion over $K$ which has maximal
adelic Galois image, but is not semistable.
\begin{ex}\label{nonsemi_ex}
Let $\beta$ be the unique real root of $x^{3} + 4x^{2} + 7x - 4$ and $K = \Q(\beta)$.
The adelic representation $\rho_{E}$ associated to
the elliptic curve $$E/K \colon y^{2} = x(x + (10\beta^{2} - 3))(x - (\beta + 4))$$
surjects onto $V_1(2)  \times \prod_{\ell > 2} \gl_2(\mathbb{Z}_{\ell})$.
\end{ex}
Lastly, we give an example of non-galois cubic field $K$ with nontrivial narrow class group and a semistable elliptic curve $E/K$
with surjective adelic representation.
\begin{ex}\label{nontriv_ex}
Let $\beta$ be the unique real root of $x^{3} + 8x^{2} - 3x + 1$ and $K = \Q(\beta)$. Then the adelic representation $\rho_{E}$ corresponding to
$$E/K \colon y^{2} + xy + \beta y = x^{3} - 8x^{2} - 6x - 1$$ surjects onto $\gl_{2}(\widehat{\mathbb{Z}})$.
\end{ex}


While we focus on two-torsion in the paper, our tests may be easily generalized to higher torsion. In the final section, we discuss how to compute the $\ell$-adic image in the presence of $\ell$-torsion.    As a demonstration, we compute the $7$-adic Galois representation for an elliptic curve over $\mathbb{Q}$ with a single cycle of $7$-torsion points over $\mathbb{Q}$. In this case, the maximum possible image of Galois is the pre-image of the ``half-Borel" subgroup
\[
\begin{pmatrix} 1 & * \\ 0 & * \end{pmatrix} \subset \GL_2(\mathbb{F}_7)
\]
under the projection $GL_2(\adelezz) \rightarrow \GL_2(\mathbb{F}_7)$.

\begin{ex}\label{tor_ex}
Let $$E/\Q \colon y^{2} + xy + y = x^{3} - x^{2} - 19353x + 958713,$$ which has torsion subgroup
$\Z/7\Z$ over $\Q$.  Then $E$ has maximal adelic image. 
\end{ex}

The computations in this paper were done using \textsc{Magma} and \textsc{Sage}.
Throughout this paper, we include references to the code that we have written.
We indicate this by using the \texttt{typewriter} font. These scripts along with their outputs are available at \cite{transcript}.


\subsection*{Acknowledgements}
The authors wish to thank Professors Ken Ono and David Zureick-Brown for giving
us guidance and support during our research experience at the 2012 Emory University
REU. We are also grateful to the NSF and the Asa Griggs Candler Fund for their
support. Finally, we would like to thank the anonymous referee for the valuable suggestions
on improving this paper.

\section{The Criteria}\label{criteria}
We begin with the proof of Theorem \ref{conditions}.
\subsection{Proof of Theorem \ref{conditions}}\label{maincrit}
Throughout, we will assume that our elliptic curves have full 2-torsion over $K$, i.e., the mod 2 image of the representation is trivial. As mentioned in the introduction,
\[
\rho_E(G_K) \subset  V_1(2) \times \prod_{\ell > 2} \GL_2(\mathbb{Z}_{\ell}).
\]
Let $$G := V_1(2) \times \prod_{\ell > 2} \GL_2(\mathbb{Z}_{\ell}).$$ In this section we develop a criterion for proving that the Galois representation of an elliptic
curve with full $2$-torsion surjects onto $G$.

Our approach is modeled on that in \cite{greicius}, in which the author proves a criterion for an elliptic curve to have surjective
adelic Galois representation and then uses it to compute the first known example of such
a curve. To do so, he first develops a fairly abstract criterion for a subgroup of a product of profinite groups to
be the whole group, and then interprets it in the case of a Galois representation associated
to an elliptic curve.  We summarize this result in the following lemma.

\begin{lemma}[\cite{greicius}, Lemma 2.2 and Proposition 2.5]
Let $G_{\alpha}$ be a collection of profinite groups such that for $\alpha \neq \alpha'$,
we have $\Quo(G_\alpha) \cap \Quo(G_{\alpha'}) = \varnothing$, where $\Quo(G_\alpha)$
denotes the set of isomorphism classes of finite, nonabelian simple quotients of $G_\alpha$.
If $H \subseteq G = \prod_{\alpha} G_\alpha$ is a closed subgroup that surjects onto each
$G_\alpha$ and surjects onto $G^{\ab}$, then $H=G$.
\end{lemma}

Applying this to our situation, we know that $$\rho_{E}(\gk) \subseteq V_1(2) \times \prod_{\ell > 2} \GL_2(\mathbb{Z}_{\ell})$$
is a closed subgroup. The condition on $\Quo(G_\alpha)$ is easily satisfied;
we are in the same situation as Corollary 2.7 of \cite{greicius}, except now we have replaced
$\GL_2(\Zb_2)$ with $V_1(2)$. In fact, $\Quo(V_1(2)) = \varnothing$, as
$V_1(2)$ is a pro-$2$ group, so the condition is still satisfied. Therefore, in order to show that $\rho_{E}(\gk)$ is the entire group, we must show the
$\rho_{E, \ell^\infty}$ is surjective for $\ell \neq 2$, that $\rho_E(\gk)$ surjects onto
$V_1(2)$, and that $\rho_E(\gk)$ surjects onto the abelianization of
$V_1(2) \times \prod_{\ell > 2} \GL_2(\mathbb{Z}_{\ell}).$

In order to check the first two conditions, we apply the following two lemmas which are known as the Refinement Lemmas.

\begin{lemma}[\cite{lang_trotter}, p. 47]\label{langtrotter}
Let $U \subseteq V_1(\ell)$ be a closed subgroup. Then the following are true.
\begin{enumerate}
\item If $\ell=2$ and $U$ surjects onto $V_{1}(\ell)/V_{3}(\ell)$,
  then $U=V_1(\ell)$.
\item If $\ell$ is odd, and $U$ surjects onto $V_{1}(\ell)/V_{2}(\ell)$, then $U=V_{\ell}$.
\end{enumerate}
\end{lemma}

\begin{lemma}[\cite{serre98}, IV-23]\label{serre}
Let $\ell \geq 5$. Suppose $H \subseteq \SL_2(\Zl)$ is a closed subgroup that surjects onto
$\SL_2(\Fl)$. Then $H=\SL_2(\Zl)$.
\end{lemma}

\begin{cor}\label{ladiccriteria}
Suppose that $K \cap \Qb^{\cyc} = \Qb$, and suppose the following conditions are satisfied:
\begin{enumerate}
\item $\rho_{E,2^\infty}(\gk)$ surjects onto $V_1(2)/V_3(2)$.
\item $\rho_{E,9}(\gk)$ is surjective.
\item $\rho_{E,\ell}(\gk)$ is surjective for all $\ell > 3$.
\end{enumerate}
Then $\rho_{E,2^\infty}$ surjects onto $V_1(2)$, and $\rho_{E,\ell^\infty}$ is surjective
for all odd $\ell$.
\end{cor}

\begin{proof}
The surjectivity for the $2$-adic and $3$-adic representations follows from the two parts of
Lemma \ref{langtrotter}, respectively.

Furthermore, if we assume that $K \cap \Qb^{\cyc} = \Qb$, then, by the non-degeneracy of the Weil pairing, the
determinant map is surjective. The hypothesis in our corollary implies that $\rho_{E,\ell}(\gk)$
contains $\SL_2(\Fl)$, so $\rho_{E,\ell^\infty}(\gk)$ contains $\SL_2(\Zl)$. As the determinant is
surjective, this implies that $\rho_{E,\ell^\infty}$ is surjective.
\end{proof}

In order to show that $\rho_{E}(\gk)$ surjects onto the abelianization, we must compute the
abelianization of $V_1(2) \times \prod_{\ell > 2} \GL_2(\mathbb{Z}_{\ell}).$ Equivalently,
we wish to compute its commutator subgroup. We note that we only need
to do this for each $\ell$ separately. For any group $G$, we denote
the commutator subgroup of $G$ by $G'$. We recall the computation for $\ell$ odd.

\begin{lemma}[\cite{lang_trotter}, p. 95 and 183]
Let $\ell$ be odd. Then $\GL_2(\Zl)'=\SL_2(\Zl).$
\end{lemma}

The only task then is to compute $V_1(2)'$.

\begin{lemma}
We have $V_1(2)' =   V_2(2) \cap \SL_2(\mathbb{Z}_2).$
\end{lemma}

\begin{proof}
As $V_1(2)/V_2(2)$ is abelian, we know the commutator is contained in $V_2(2)$. As usual, we know
that $\SL_2(\Zb_2)$ contains the commutator, so we have an inclusion above in one direction. We write
$U_k(\ell)= V_k(\ell) \cap \SL_2(\Zl)$.

First, we would like to show that modulo $8$, the image of $V_1(2)'$ contains $U_2(2)$.
In other words, we would like to compute the commutator of $V_1(2)/V_3(2)$.

We note that
$$\left[\pmat{3}{0}{2}{3}, \pmat{3}{0}{2}{1}\right] = \pmat{1}{0}{4}{1} \quad \text{and} \quad \left[\pmat{3}{0}{0}{1}, \pmat{1}{2}{2}{1}\right] = \pmat{1}{4}{4}{1}.$$
Transposing the first, we get the matrix $\smat{1}{4}{0}{1}.$
These three matrices generate $U_2(2)/U_3(2)$, so we know that this is the commutator of $V_1(2)/V_3(2)$, and this means that $V_1(2)'$ surjects onto $U_2(2)/U_3(2)$.

Suppose that $V_1(2)'$ surjects onto $U_{k-1}(2)/U_k(2)$. Note that
$U_k(\ell) = I + \ell^k \mf{sl}_2(\mathbb{Z}_{\ell})$, where $\mf{sl}_2$ is the set of $2 \times 2$ matrices with trace $0$. We know that $V_1(2)'$ contains all matrices of the form $1+2^{k-1}A$, as $A$ runs over the matrices of trace zero modulo $2$. In general, if $ k \geq 2$ or $\ell \geq 3$, we know that $(1+\ell^k A)^2 \equiv 1+\ell^{k+1}A \imod{\ell^{k+2}}$,
and this only depends on $A$ modulo $\ell$. It follows that $V_1(2)'$ also surjects onto $U_k(2)/U_{k+1}(2)$. Since we know that $V_1(2)'$ surjects onto $U_2(2)/U_3(2)$, we know that it surjects
onto $U_k(2)/U_{k+1}(2)$ for all $k \ge 2$. Thus $V_1(2)'$ surjects onto $U_2(2)/U_k(2)$, and since
$V_1(2)'$ is closed, it must be all of $U_2(2)$.
\end{proof}

So we see that the abelianization map of $G$ is the product
$\rho_{E,4} \times \det$, and its image is the index two subgroup of $V_2/V_4 \times \adelezz^* $
where $\det$ (on the first factor) and reduction mod $2$ (on the second factor) agree.

We would like to interpret surjectivity onto $G^{\ab}$ in a field-theoretic manner. As we noted,
showing that the determinant map is surjective is a matter of requiring that $K \cap \Qb^{\cyc} = \Qb$.
In order for $\rho_E(\gk) $ to surject onto $G^{\ab}$ in our case, we need the following:

\begin{prop}
The group $\rho_E(\gk)$ surjects onto $G^{\ab}$ if and only if
$K \cap \mathbb{Q}^{\cyc} = \mathbb{Q}$, $[K(E[4]):K] = 16$, and $K(E[4]) \cap
K^{\cyc} = K(i)$.
\end{prop}

\begin{proof}
We have already remarked that $K \cap \mathbb{Q}^{\cyc} = \mathbb{Q}$ is equivalent to the surjectivity of the determinant map.

Observe that $K(E[4])$ is the fixed field of $\rho_{E,2^{\infty}}(G_K) \cap (I+4V_1(2))$. Furthermore,  $K(E[4])  \cap K^{\cyc} \supset K(i)$
by properties of the Weil pairing.
The abelianization map can be restated as the product of the quotient maps from $G_K$
to $\Gal(K(E[4])/K)$ and $\Gal(K^{\cyc}/K)$ given by restriction to the respective fields;
under this interpretation, the condition is simply that the image of an element
$\sigma \in G_K$ must agree under the restriction of both factors to $K(i)$.

We know that if $L$ and $M$ are extensions of $K$, then $\Gal(LM/K)$ is $\Gal(L/K) \times_{\Gal(L \cap M)/K} \Gal(M/K)$, and the result follows. Therefore, if $K(E[4]) \cap K^{\cyc} = K(i)$, then the Galois group of the compositum already surjects onto $G^{\ab}$.  Conversely, if the intersection is any greater, then the map cannot be surjective because its image is
constrained to agree on the Galois group of the intersection over $K$, which strictly contains $K(i)$.
\end{proof}

\subsection{Checking condition \ref{intersection}}
As the second condition of Theorem \ref{conditions} cannot be directly computed, we must develop an easier criterion to check.
Greicius does not have to do this, as his version of the second condition is that $K(\sqrt{\Delta}) \cap \Kcyc = K$, so all that must be checked is if $\sqrt{\Delta} \in \Kcyc$.  We avoid working with $K(E[4])$ directly, as that would be computationally intensive.  Instead, we develop a test in which we only see if certain elements of $K$ are squares.

We use the following notation.  Let $T$ be a set of numbers.
Define $\mathcal{P}_T$ to be the set containing all products of non-empty subsets of elements of $T$.  For example,
if $T = \{2, 3\}$, then $\mathcal{P}_T = \{2, 3, 6\}$.  We define $\Delta$ to be the discriminant of $E$.

\begin{prop}\label{intersection_thm}
Assume that $[K(E[4]): K] = 16$ and $K \cap \Q^{\cyc} = \Q$.  Let $d_1, d_2, d_3$ be the discriminants of the three quadratic factors of the 4-division polynomial.  Let $d_4 = \sqrt{\Delta}$.  Let $S'$ be the set of prime ideals of $K$ dividing $2 \Delta$ and $S$ be the primes of $\Q$ below primes in $S'$.  Let $T = \{d_1, d_2, d_3, d_4\}$.
 If $K(E[4]) \cap \Kcyc \supsetneq K(i)$,
then there exists $s \in \calP_S$, $t \in \calP_T$ such that $t/s$ is a square in $K$.
\end{prop}

\begin{proof}
From p. 81-82 of \cite{torsion}, if $E$ is in the form $y^2 = x^3 + A x + B$, then $K(E[4]) = K(x(E[4]), \sqrt[4]{\Delta}) = K(\sqrt{d_1}, \sqrt{d_2}, \sqrt{d_3}, \sqrt{d_4})$.
Our curves are not of that form, but using the change of variables from pg. 41 - 43 of \cite{silverman1}, the same result holds for all elliptic curves in Weierstrass form.

Assume that $K(E[4]) \cap \Kcyc = M' \supsetneq K(i)$. Then $\Gal(K(E[4])/K) \cong (\Z/2\Z)^4$,
so there exists $M$ with $K(i) \subset M \subset M'$ and $\Gal(M/K)
\cong (\Z/2\Z)^2$.   The extension $M/K$ has
two quadratic subextensions in addition to $K(i)$.  Let $L := K(\sqrt{d})$, where $d \in \calO_K$ is square-free, be one of them.

As $L \subset \Kcyc$, choose $n$ such that $L \subset K(\mu_n)$.   Since $K  \cap \Q^{\cyc} = \Q$,
$\Gal(K(\mu_n)/K) \cong \Gal(\Q(\mu_n)/\Q)$.  So we may consider $\Q \subsetneq L' := \Q(\mu_n)^{\Gal(K(\mu_n)/L)} \subset \Q(\mu_n).$
By the fundamental theorem of Galois theory, $L' \subset L$.  So $L \cap \Q^{\cyc} \neq \Q$.
This implies that $\sqrt{d} \in \Q^{\cyc}$, and hence $d \in \Q^{\cyc} \cap \calO_K = \Z$.  We may assume that $d$ is positive, as $i \in M$.

By the criterion of Neron-Ogg-Shafarevich, a prime ideal $\frakp \subset \calO_K$ ramifies in
$K(E[4])/K$ only if $\frakp \in S'$.  Therefore, we must have that $d \in \calP_S$.

Since $K(E[4]) = K(\sqrt{d_1}, \sqrt{d_2}, \sqrt{d_3}, \sqrt{d_4})$, all quadratic subextensions of
$K(E[4])/K$ are generated by an element of $\calP_T$. Therefore, $L = K(\sqrt{t})$ for some $t \in \calP_T$.  Now, $K(\sqrt{t}) = K(\sqrt{s})$ if and only if $t/s$ is a square of $K$, which completes the proof of our lemma.
\end{proof}

\subsection{2-adic surjectivity}\label{2adicsurj}
We now outline our algorithm for checking that the $2$-adic image is the full group $V_1(2)$. We assume
that the $2$-adic image surjects onto $V_1(2)/V_2(2)$; this is easily checked by, for instance, computing
the degree of the field extension $K(E[4])/K$.  In contrast, computing directly the degree of the extension $K(E[8])/K$ is infeasible, as we want examples where this extension has degree 256.

In what follows, let $\mf{p}$ be a prime of $K$ not dividing $2 \Delta$, where $\Delta$ is the
discriminant of $E$ (so that the $2$-adic Galois representation is unramified at these primes).
Let $\wt{E}_{\mf{p}}$ denote the reduction of $E$ at $\mf{p}$. Let $\Frob_{\mf{p}}$
denote the Frobenius element at $\mf{p}$ for the extension $\overline{K}/K$.

\begin{prop}\label{2adicprop}
Under the hypothesis above, the 2-adic Galois representation
associated to $E$ has maximal image if the following is satisfied:
there exists a prime $\mf{p} \subset \mathcal{O}_K$ such that $\#(\mathcal{O}_K/\mf{p}) \equiv 5 \imod{8}$ and $E \imod{p}$
has full four-torsion over $\mathcal{O}_K/\mf{p}$.
\end{prop}

\begin{proof}
By the Refinement Lemmas (\cite{lang_trotter}, p. 47), it suffices to show that the $2$-adic image surjects onto $V_1(2)/V_3(2)$. Since we know that it surjects
onto $V_1(2)/V_2(2)$, the mod $8$ image contains an element of the form $I + 2S + 4T$ for all $2 \times 2$ matrices $S \in \text{Mat}_{2 \times 2}(\mathbb{Z}/2\mathbb{Z})$. Squaring, we obtain
\[
(I+2S+4T)^2 \equiv I + 4(S + S^2) \imod{8}.
\]
Computing this for all such $S$, we find five distinct matrices:
\[
\begin{pmatrix} 1 & 0 \\ 0 & 1\end{pmatrix}, \begin{pmatrix} 1 & 4 \\ 0 & 1\end{pmatrix}, \begin{pmatrix} 1 & 0 \\ 4 & 1\end{pmatrix}, \begin{pmatrix} 5 & 4 \\ 4 & 5\end{pmatrix}, \begin{pmatrix} 5 & 0 \\ 0 & 5\end{pmatrix}.
\]
Note that this set generates a subgroup of size $8$ in $V_2(2)/V_3(2)$ consisting of those matrices with determinant congruent to $1 \imod{8}$.
Therefore, we immediately see that (under our hypotheses) the image of Galois, contains an index two subgroup of
$V_2(3)/V_2(3)$ and hence also $V_1(3)/V_3(3)$. Moreover, equality is achieved in both cases if and only if
the mod $8$ image of Galois has an element with determinant congruent to $5 \imod{8}$ in $ V_2(2)/V_3(2)$.

Since the $8$-torsion of $E$ injects into the $8$-torsion of $\wt{E}_{\mf{p}}$, the first condition implies that in some basis for $E[8]$,
\[
\rho_{E,4}(\Frob_{\mf{p}})  = I \implies \rho_{E,8}(\Frob_{\mf{p}})  \equiv I \imod{4}
\]
but has determinant congruent to $5 \imod{8}$.
\end{proof}

\subsection{Proof of Example \ref{single_example}}
In this section we apply the results of Section \ref{criteria} and give an example of an elliptic curve over $K = \Q(\alpha)$ with surjective
adelic Galois representation $\rho_{E}$.

Let
$$E/K: y^{2} = x(x - (2\alpha^{2} + 7\alpha + 19))(x - (18\alpha^{2} + 7\alpha + 3)).$$
We want to show that $\rho_{E}$ surjects onto $V_1(2)  \times \prod_{\ell > 2} \gl_2(\mathbb{Z}_{\ell})$.

It is clear from the definition of $E$ that $E$ has full $2$-torsion over $K$.
As the conductor of $E$ is squarefree, $E$ is semistable.
We now check that the three conditions of Theorem \ref{conditions}
are satisfied. Since $K = \Q(\alpha)$, it is clear that $K \cap \Q^{\cyc} = \Q$.

As an ideal the discriminant of $E$ is
\begin{align}\label{exeq1}
(\Delta) = (2)^{12}(2\alpha^{2} + 7\alpha + 19)^{2}(5\alpha^{2} + \alpha + 4)^{2}(\alpha - 3)^{2}(\alpha - 1)^{2}
\end{align}
where the ideals on the right hand side of the above equation are prime ideals which lie
above 2, 6857, 167, 31, and 3, respectively.

Using $\textsc{Magma}$, one can compute that $[K(E[4]): K] = 16$.
We now show that $K(E[4]) \cap K^{\cyc} = K(i)$. This will make use of Lemma \ref{intersection_thm}.
To apply the theorem, we make some computational preliminaries to explicitly write out $K(E[4])$.
Note that the division polynomial of $E$ will factor as the product of 3 quadratic factors $f_{1}$, $f_{2}$, and $f_{3}$.
From Section 5.5 of \cite{torsion} and the fact that $[K(E[4]) : K] = 16$, we know that
$K(E[4]) = K(\sqrt{d_{1}}, \sqrt{d_{2}}, \sqrt{d_{3}}, \sqrt[4]{\Delta})$
where $d_{i}$ is the discriminant of $f_{i}$. In this case, we have
\begin{align*}
d_{1} &= 361\alpha^{2} - 22\alpha - 83 \\
d_{2} &= -2112\alpha^{2} - 2048\alpha - 640 \\
d_{3} &= -372\alpha^{2} + 1976\alpha + 1552.
\end{align*}
We also have that
\[d_{4} := \sqrt{\Delta} = 2^{6}\alpha^{3}(2\alpha^{2} + 7\alpha + 19)(18\alpha^{2} + 7\alpha + 3)(\alpha - 1).\]
In the notation of Lemma \ref{intersection_thm}, we have $S = \{2, 3, 31, 167, 6857\}$ and $T = \{\sqrt{d_{i}}\}_{i = 1}^{4}$.
Let $$S_{1} = \{2^{r_{1}}3^{r_{2}}31^{r_{3}}167^{r_{4}}6857^{r_{5}} : r_{i} = 0 \text{ or } 1\} \backslash \{1\}$$
and $$T_{1} = \{d_{1}^{s_{1}}d_{2}^{s_{2}}d_{3}^{s_{3}}d_{4}^{s_{4}} : s_{i} = 0 \text{ or } 1\} \backslash \{1\}.$$
Lemma \ref{intersection_thm} implies that if for all possible pairs of $\beta \in S_{1}$ and $\gamma \in T_{1}$
if $\gamma/\beta$ is not a square in $K$, then $K(E[4]) \cap K^{\cyc} = K(i)$. Writing a script in \textsc{Magma}
iterating through all $465$ pairs of $(\beta, \gamma)$ yields that this condition is indeed satisfied
(see \texttt{ex12\_sq.txt} and the corresponding \texttt{ex12\_sq\_res.txt} in \cite{transcript}). Therefore we have
$K(E[4]) \cap K^{\cyc} = K(i)$.

Let $H := \rho_{E}(G_{K})$. Denote by $H_{\ell}$ the image of $H$ under $\pi_{\ell} : H \rightarrow \gl_{2}(\Z_{\ell})$,
and similarly denote by $H(\ell)$ the image of $H$ under $r_{\ell} : H \rightarrow \gl_{2}(\Z/\ell\Z)$.
Since $[K(E[4]) : K] = 16$, the mod 4 image $H(4)$ is maximal.
This and the fact that $K(E[4]) \cap K^{\cyc} = K(i)$ imply that the abelianization map is surjective. We will also need this fact in proving that
$\rho_{E, 2^{\infty}}$ is surjective.

We now verify the first condition in Theorem \ref{conditions}. That is, $H_{\ell} = \gl_{2}(\Z_{\ell})$ for $\ell \geq 3$
and $H_{2} = V_{1}(2)$. The ideal generated by the $j$-invariant of
$E$ is
$$(j_{E}) = \frac{(\alpha^{2} - 6\alpha + 3)^{3}(3\alpha^{2} - 5\alpha + 1)^{3}(5\alpha^{2} + 2\alpha + 6)^{3}}{(5\alpha^{2} + \alpha + 4)^{2}(\alpha-1)^{2}(\alpha - 3)^{2}(2\alpha^{2} + 7\alpha + 19)^{2}}.$$
Equation \eqref{exeq1} gives the primes of bad reduction.
For all primes of bad reduction which are not the ideal $(2)$, we have $v(j_{E}) = -2$. To show that $H_{\ell} = \gl_{2}(\Z_{\ell})$
for all $\ell \neq 2, 3, 31$, we will make use of a combination of Proposition 3.5, Corollary 3.8, and
Remark 3.9 in \cite{greicius} which we state below:
\begin{prop}\label{greiciusprop}
Let $K$ be a number field with a real embedding and a trivial narrow class group and satisfying $K \cap \Q^{\cyc} = \Q$.
Let $E/K$ be a semistable elliptic curve with $j$-invariant $j_{E}$. Suppose $\ell$ is a prime unramified in $K$. If
$\ell = 2, 3, 5$, suppose further that $\ell \nmid v(j_{E})$
for some $v$ corresponding to a prime of bad reduction of $E$. If $H(\ell) \neq \gl_{2}(\F_{\ell})$, given any place $v'$ not of bad reduction,
we have $\ell \mid \#\wt{E}_{v'}(k_{v'})$ where $\wt{E}_{v'}(k_{v'})$ is the reduction of $E$ over the residue field corresponding to $v'$.
\end{prop}

Fix an $\ell \neq 2, 31$. Assume that $H(\ell) \neq \gl_{2}(\F_{\ell})$. Let $v$ be the prime ideal
$(1 - 2\alpha)$ which lies over 13. Then $\card\wt{E}_{v}(k_{v}) = 16$ and hence by Proposition \ref{greiciusprop}, $\ell \mid 16$.
Since $\ell$ is prime, we must have $\ell = 2$, a contradiction. Therefore $H(\ell) = \gl_{2}(\F_{\ell})$ for all $\ell \neq 2, 31$.

We now use the following result which is a corollary of Lemma \ref{serre}.
\begin{prop}[\cite{greicius}, Corollary $2.13$ (iii)]\label{refinecor}
Let $K \subseteq \gl_{2}(\Z_{\ell})$ be a closed subgroup. If $\ell \geq 5$,
$K \twoheadrightarrow \gl_{2}(\F_{\ell})$ and $\det(K) = \Z_{\ell}^{\ast}$,
then $K = \gl_{2}(\Z_{\ell})$.
\end{prop}
For $\ell \neq 2, 3, 31$, the previous proposition and the surjectivity of $\det|_{H}$ imply that $H_{\ell} = \gl_{2}(\Z_{\ell})$
for all $\ell \neq 2, 3, 31$. It now remains to show the last 3 cases.

Consider the $\ell = 31$ case. This case will make use of the following proposition.
\begin{prop}[\cite{serre}, Proposition 19]\label{serreprop}
Let $\ell \geq 5$ and suppose $K \subseteq \gl_{2}(\F_{\ell})$ is a subgroup satisfying
\begin{enumerate}[(i)]
\item $K$ contains elements $s_{1}$ and $s_{2}$ such that
$\left(\frac{\tr(s_{i})^{2} - 4\det(s_{i})}{\ell}\right) = (-1)^{i}$ and $\tr(s_{i}) \neq 0$.
\item $K$ contains an element $t$ such that $u = \tr(t)^{2}/\det(t) \not\equiv 0, 1, 2, 4\imod{\ell}$
and $u^{2} - 3u + 1 \not\equiv 0 \imod{\ell}$.
\end{enumerate}
Then $K$ contains $\operatorname{SL}_{2}(\F_{\ell})$. In particular, if $\det : K \rightarrow \F_{\ell}^{\ast}$
is surjective, then $K = \gl_{2}(\F_{\ell})$.
\end{prop}
We compute the following table. By $t_{v}$ and $N_{v}$, we mean the trace of the Frobenius corresponding to $v$ and the
cardinality of the residue field corresponding to $v$.
\begin{align*}
\begin{array}{c|c|c|c}
  v  & t_{v}^{2} - 4N_{v}^{2} \imod{31} & u = t_{v}^{2}/N_{v} \imod{31} & u^{2} - 3u + 1 \imod{31} \\\hline
  \mf{p}_{11} = (\alpha - 2) & 3 & 24 & 9 \\
  \mf{q}_{11} = (-5\alpha^{2} + \alpha - 3) & 17 & 19 & 26 \\
  \mf{q}_{13} = (-2\alpha + 1) & 14 & 17 & 22
\end{array}
\end{align*}
Let $s_{1} := \rho_{E, 31}(\frob_{w_{1}})$ for any fixed $w_{1} \mid \mf{p}_{11}$, $s_{2} := \rho_{E, 31}(\frob_{w_{2}})$
for any fixed $w_{2} \mid \mf{q}_{13}$, and $t := \rho_{E, 31}(\frob_{w_{3}})$ for any fixed $w_{3} \mid \mf{q}_{11}$.
In Proposition \ref{serreprop}, take $K := H(31)$. By the above table and Proposition \ref{serreprop},
$H(31)$ contains $\operatorname{SL}_{2}(\F_{31})$. Since $\det: H(31) \rightarrow \F_{31}^{\ast}$ is surjective,
we have $H(31) = \gl_{2}(\F_{31})$. Then by Proposition \ref{refinecor}, $H_{31} = \gl_{2}(\Z_{31})$.

Now consider the $\ell = 3$ case.
Consider the prime ideal $\mf{q}_{29} := (3\alpha^{2} + 2)$ which lies over 29 and let $\pi \in H_{3}$ be $\rho_{E, 3}(\frob_{w})$
for any $w \mid \mf{q}_{29}$. Then for $v = \mf{q}_{29}$, $N_{v} = 29$ and $t_{v} = 6$ and hence the characteristic polynomial
of $\pi$ is $t^{2} - 6t + 29 \equiv (t - 7)(t - 8) \imod{9}$. The rest of the computation for this case remains exactly the same
as the $\ell = 3$ case considered by Greicius in \cite{greicius}. Therefore $H_{3} = \gl_{2}(\Z_{3})$.

Finally we consider the $\ell = 2$ case. Recall that since $[K(E[4]) : K] = 16$, $H(4)$, the mod 4 image is maximal.
We will now use Proposition \ref{2adicprop}. Letting $\mf{p} := (157)$ in the proposition
gives that $\card(\mathcal{O}_{K}/\mf{p}) = 3869893 \equiv 5 \imod{8}$. Computation in $\textsc{Sage}$ shows that
$E$ has full four-torsion over $\mathcal{O}_K/\mf{p}$. This verifies that $H_{2} = V_{1}(2)$.

The three conditions of Theorem \ref{conditions} are now verified as desired.

\begin{rem}
We have written code in \textsc{Sage} that automates the checking we did above and gives the prime ideals that one needs to satisfying the
given conditions (see \texttt{check\_sage.sws} in \cite{transcript}).
\end{rem}

\section{Constructing an Infinite Family}\label{infinite}
Fix $K = \mathbb{Q}(\alpha)$ where $\alpha$ is the real root of $x^3+x+1$. In this section, we apply our criteria to obtain
an \emph{infinite} family of elliptic curves with maximal adelic image given full two-torsion over $K$. To do so, we obtain families
of curves satisfying each of the requisite conditions:
\begin{enumerate}[(i)]
\item semistability,
\item $K(E[4]) \cap K^{\cyc} = K(i)$, and
\item maximal $\ell$-adic image for each prime $\ell$.
\end{enumerate}

We are able to find congruence conditions on the coefficients of our elliptic curves ensuring each of these conditions, which we
may then piece together using the Chinese remainder theorem.

Throughout this section, with $b, c \in \Z$, let $E_{b,c}$ be the elliptic curve defined by
\[
E_{b,c}/K :  y^2 = x(x-(\alpha^2 + b\alpha + c))(x-16(\alpha^2+\alpha+1)).
\]

\subsection{A family of semistable elliptic curves}
Since the criteria we use for verifying maximality apply only to semistable curves,  we must first find
infinite families of semistable elliptic curves of the desired form. Such a family is furnished by the following proposition.

\begin{prop}
Suppose that one of the following two conditions is satisfied:
\begin{enumerate}[(i)]
\item $b \equiv 5 \imod{12}$ and $c \equiv 4 \text{ or } 8 \imod{12}$.
\item $b \equiv 9 \imod{12}$ and $c \equiv 4 \imod{12}$.
\end{enumerate}
Then
\[
E/K: y^2 = x(x-(\alpha^2 + b \alpha + c)) (x-16(\alpha^2 + \alpha + 1))
\]
is semistable.
\end{prop}

\begin{proof}
Since $\mathcal{O}_{K}$ is a unique factorization domain, we will speak freely of primes of $K$ as numbers. We establish the result for
case $(i)$. The proof for case $(ii)$ is identical.

Let $A = \alpha^2 + b \alpha + c$, $B = 16(1 + \alpha + \alpha^2)$, and $C = A-B$. We first claim that $A$ and $B$ are coprime in
$K$. Since $\alpha^3 + \alpha + 1 =0$,
$1 + \alpha + \alpha^2 = \alpha^2 - \alpha^3 = \alpha^2(1 - \alpha)$.
Moreover, $\alpha$ is a unit. Furthermore, $N_{K/\Q}(1 - \alpha) = f(1) = 3$.
So $1 + \alpha + \alpha^2$ is a prime lying over $3$.

We also easily check that $2$ is inert in $K$. Hence it suffices to prove that $A$ is coprime to 2 and $1 - \alpha$.
The former is clear. For the latter, observe that
$$A = \alpha^{2} + b\alpha + c \equiv 1 + b + c \imod{1 - \alpha}.$$
But $b \equiv 2 \imod{3}$ and $c \equiv 1 \text{ or } 2 \imod{3}$, so $1 + b + c \not\equiv 0 \imod{3}$.

So we see that $A$ and $B$ are coprime, hence $\{A,B,C\}$ are pairwise coprime. Now we note that for this Weierstrass form,
\begin{align*}
\Delta = 16 (ABC)^{2} \quad \text{ and } \quad c_4 = 16(A^2 - AB + B^2)
\end{align*}
If $\mf{p} \neq (2)$ is a prime dividing $\Delta$, then it divides exactly one of $A,B,C$. Supposing that $\mf{p} \mid A$,
we see that $c_4 \equiv 16B^2 \not\equiv 0 \imod{\mf{p}}$, so $E$ is not semistable at $\mf{p}$. Similarly, if $\mf{p} \mid B$
the curve is semistable at $\mf{p}$. Finally, if $\mf{p} \mid C = A-B$, then
\[
c_4 = 16  (A^2-AB+B^2) \equiv 16A^2  \not\equiv 0 \imod{\mf{p}}.
\]
The only prime left to consider is $\mf{p}=(2)$. Making the change of variables $(x,y) \mapsto (4x, 8y + 4\alpha^2 x)$,
and noting that $\alpha^4 = -\alpha(\alpha+1)$, we obtain the Weierstrass equation
$$y^{2} + \alpha^{2}xy = x^{3} + \frac{-A -B + \alpha(\alpha + 1)}{4}x^{2} - \frac{AB}{16}x.$$
Since $A = \alpha^{2} + b \alpha + c$, we see that $A -\alpha - \alpha^2 $ is divisible by $4$, and hence the coefficients
are in $\mathcal{O}_K$. For this Weierstrass form, $\Delta = (ABC)^{2}/16$ and $c_{4} = A^{2} - AB + B^{2}$.
Since $2 \mid B$ and $2 \nmid A$, we see that $E$ is semistable at $2$ as well.
\end{proof}

\subsection{A family with maximal mod $4$ image}\label{family_mod4}
In this section, fix an elliptic curve $E_{b, c}$.
To verify maximal $2$-adic image, we must first have maximal mod $4$ image and then apply Proposition \ref{2adicprop}. That
$\rho_4(G_K)$ is maximal is equivalent to $[K(E[4]):K] = 16$.
We will describe a procedure that may be used to determine congruence conditions on $b$ and $c$ which ensure
that $E$ has maximal mod $4$ image. As an example, we prove the following:

\begin{prop}\label{mod4prop}
Let
\[
E/K \colon y^2 = x(x-(\alpha^2 + b\alpha + c))(x-16(\alpha^2+\alpha+1)),
\]
where
$b \equiv 3699 \imod{4991}$ and $c \equiv 4183 \imod{4991}.$
Then $\rho_{E,4}(G_K)$ is maximal.
\end{prop}

We emphasize that the congruence condition given in the proposition is just one of many possibilities.
We begin with a general discussion. Recall that
\[
K(E[4]) = K(x_1, x_2, x_3, \sqrt[4]{\Delta})
\]
where $x_1, x_2, x_3$ are the $x$-coordinates of $4$-torsion points not differing by a $2$-torsion point. We
may take $x_1, x_2, x_3$ to be the roots of the quadratics
\begin{align*}
x^2 &- 16(b+c) \alpha^2 - 16(c-2)\alpha - 16(c-b-1)   \\
x^2 &- 32 (\alpha^2+\alpha+1)x + 16( (b+c)\alpha^2 + (c-2)\alpha + (c-b+1))  \\
x^2 &+ (-2\alpha^2 - 2b\alpha - 2c)x + 16(b+c)\alpha^2 + 16(c-2)\alpha + 16(c-b-1)
\end{align*}
which we obtained by factoring the $4$-division polynomial for $E$, whose discriminants are, respectively,
\begin{align*}
d_1 &= 2^6 ((b+c)\alpha^2 + (c-2)\alpha+ (c-b-1)) \\
d_2 &= 2^6 \left(16  (1+\alpha+\alpha^2)^2 -  ( (b+c)\alpha^2 + (c-2)\alpha + (c-b-1)) \right) \\
d_3 &= 4(\alpha^2 + b\alpha + c)^2 - 64((b+c)\alpha^2 + (c-2)\alpha + (c-b-1)).
\end{align*}
In addition,
\[
d_4:= \sqrt{\Delta} = 64  (\alpha^2+b\alpha+c)  (\alpha^2+\alpha+1)  ( 15  \alpha^2 + (16-b) \alpha + (16-c)).
 \]
So
\[
K(E[4]) = K(\sqrt{d_1}, \sqrt{d_2}, \sqrt{d_3}, \sqrt{d_4}).
\]
This is a degree $16$ extension of $K$ if and only if all $15$ of the extensions $K(\sqrt{N})$, where $N$ is the
product of some non-empty subset of $\{d_1, d_2, d_3, d_4\}$, are
nontrivial, or equivalently, if and only if all such $N$ are not squares in $\mathcal{O}_K$.

So to ensure this, it suffices to find prime ideals $\mf{p}_1, \mf{p}_2, \mf{p}_{3}, \mf{p}_{4}$ such that for each
$i = 1,2,3,4$, we have $d_i$ is a non-square in $(\mathcal{O}_K/\mf{p}_i)^{\times}$ and $d_j$ is a square
for $j \neq i$. If this is the case, then the product of any non-empty subset of $d_1, d_2, d_3, d_4$ will fail to be a square in
\[
\left( \mathcal{O}_K/\mf{p}_1 \mf{p}_2 \mf{p}_{3} \mf{p}_{4} \right)^{\times} \simeq  \left(\mathcal{O}_K/\mf{p}_1 \right)^{\times} \times \left(\mathcal{O}_K/\mf{p}_2\right)^{\times} \times \left(\mathcal{O}_K/\mf{p}_3 \right)^{\times} \times \left(\mathcal{O}_K/\mf{p}_4 \right)^{\times},
\]
hence is obviously not a square in $\mathcal{O}_K$. Once ideals satisfying these conditions are found for a particular curve
\[
E : y^2 = x(x-(\alpha^2 + b\alpha + c))(x-16(\alpha^2+\alpha+1)),
\]
where $b, c \in \mathcal{O}_K$, then they will work for any curve
\[
E' : y^2 = x(x-(\alpha^2 + b'\alpha + c'))(x-16(\alpha^2+\alpha+1)).
\]
with $b \equiv b' \imod{\mf{p}_1 \mf{p}_2 \mf{p}_{3} \mf{p}_{4} } $ and $c \equiv c' \imod{\mf{p}_1 \mf{p}_2 \mf{p}_{3} \mf{p}_{4} }$.

We carried out this program by computer search (see \texttt{test\_primes.sws} in \cite{transcript}). Many examples were found, but for brevity we simply list one.

\begin{proof}[Proof of Proposition \ref{mod4prop}]
With $b = 33645$ and $c = 19156$, the prime ideals
\begin{align*}
\mf{p}_1 = (\alpha-3), \quad \mf{p}_2 = (7), \quad \mf{p}_3 = (-\alpha^2 + 3\alpha - 1), \quad \mf{p}_4 = (\alpha^2 + 4\alpha - 3)
\end{align*}
satisfy the desired conditions. These lie over 31, 7, 31, and 23 respectively, so the congruence condition is over $7 \times 23 \times 31 = 4991$.
\end{proof}

\subsection{A family with $K(E[4]) \cap \Kcyc = K(i)$}
In this section, we describe congruence conditions on $b,c$ which ensure that $E_{b,c}$ satisfies the field intersection condition
\[
K(E[4])\cap \Kcyc  = K(i).
\]
Let $d_1, d_2, d_3$, and  $d_4$ be as Section \ref{family_mod4}. Collecting terms together, we find
the following expressions.
\begin{align}\label{d3}
d_3 = (16 + 14 b - 16 c + c^2) + (31 - 16 c + 2 bc - 2b) \alpha + (b^2 - 16 b - 14 c - 1) \alpha^2 := x+y \alpha +z \alpha^2
\end{align}
and
\begin{align}\label{d4}
d_4 = -64(\alpha^2 + \alpha + 1)d_3.
\end{align}
Our key observation is that $d_4$ can be divisible by only a few rational primes, which limits the possible intersection by Lemma \ref{intersection_thm}. This is made precise in the proof of the proposition below.

\begin{prop}\label{family_intersect}
Assume that 
\[
(b,c) \not\equiv \begin{cases}  (1,1) \text{ or } (1,2) & \imod{3} \\ (1,1) & \imod{5} \\   (2,5) & \imod{11} \\
(4,5) & \imod{17} \\ (14,11) \text{ or } (23,6) \text{ or } (25, 9)& \imod{31} \\ (467, 91) & \imod{787} \\
(626, 280) & \imod{827}
\end{cases}
\]
Then $E_{b, c}$ satisfies $K(E_{b,c}[4]) \cap \Kcyc = K(i)$.
\end{prop}

\begin{proof}
Let $\Delta_{b, c}$ denote the discriminant of $E_{b, c}$.
Assume that $b, c$ are such that $K(E_{b,c}[4]) \cap \Kcyc \neq K(i)$.  Then by Lemma \ref{intersection_thm},
$K(\sqrt{n}) = K(\sqrt{m})$, where $n$ is an integer and $m$ is a product of some subset of $d_1, d_2, d_3, d_4$.
As $n$ and $m$ can only differ by a square of $K$, we have that some prime of $\Q$ divides $m$, hence some prime $p$
of $\Q$ must divide $d_1 d_2 d_3 d_4$.

Note that any prime of $K$ dividing $d_1, d_2, d_3$ must ramify in $K(\sqrt{d_i})/K$, and hence must divide $\Delta_{b, c}$
(as we already have that $2 \mid \Delta_{b, c}$). For primes of $K$ not lying above 31, dividing $\Delta_{b, c}$ is equivalent to
dividing $d_4$.   So $p \mid  d_1 d_2 d_3 d_4$ implies that either $p \mid  d_4$ or $p = 31$ and $p \mid  \Delta_{b, c}$.

Assume that $p \mid  d_4$.  Then, using \eqref{d4}, as $\alpha^2 + \alpha + 1$ is a prime above 3, either $p = 2, 3, $ or $p \mid  d_3$.  If
$p \mid  d_3$, then we must have that $p \mid  x, y, z$ (as defined in \eqref{d3}), so $p \mid  c_x x + c_y y + c_z z$ for any $c_x, c_y, c_z \in \calO_K$.
We use $c_x = 28, c_y = 8 - b, c_z = -2 + 2 c$ to get that $p \mid  698 + 377 b - 550 c.$  We use $c_x = - 2 b + 16, c_y = c - 15, c_z = 28$
to get that $p \mid  -237 - 226 b - 377 c.$ We use $c_x = -4 z - 56 c - 260, c_y = y - 28 b + 586, c_z = 56 b + 4$
to get that $ p \mid  15027 - 4844 b - 6328 c.$ Putting these three conditions together, it follows that $p \mid  618889059855$,
and then factoring yields that $p = 3$, $5$, $11$, $17$, $113$, $787$, or $827$.

Using the condition $b \equiv 1 \imod{4}, c \equiv 0 \imod{4}$, we have $2^6 \| d_1$, $2^6 \| d_2$,
$2^2 \| d_3, 2^8 \| \Delta_{b, c}$, so we can ignore powers of two when taking the square root.

Note that if for some choice of $b, c$, $p \nmid \Delta_{b, c}$, then $p \nmid \Delta_{b', c'}$ for any
$b', c' \equiv b, c \imod{p}$. Therefore, we can know which pairs $(b, c)$ have no rational primes dividing $(\Delta_{b, c})$
by simply checking if $p \mid (\Delta_{b, c})$ for $0 < b, c \leq p$ with $p = 3, 5, 11, 17, 31, 113, 787, 827$ (see \texttt{bc\_search.txt} in \cite{transcript}).
\end{proof}

Therefore, by avoiding the above values, we can easily chose congruence conditions on $b, c$ which will guarantee that
no rational prime divides $d_1 d_2 d_3 d_4$, and hence $K(E[4]) \cap \Kcyc = K(i)$.  We choose
$(b, c)$ to be congruent to $(2, 1) \imod{3}$, $(2, 4) \imod{5}$, $(6, 4) \imod{11}$, $(0, 0) \imod{17}$, and $(10, 29) \imod{31}$.

\subsection{A family with maximal $\ell$-adic image}
The tests for the remaining conditions can be easily modified into congruence conditions, as they all involve finding primes
$\frakp$ of $K$ such that $\mathcal{O}_K/\frakp$ and $E \imod{\frakp}$ satisfy certain properties, all of which do not change if we
fix the values of $b, c \imod{\frakp \cap \Z}$.  We will simply state the primes used, as the calculations are analogous to those
in Section \ref{single_example}.

\begin{prop}
Let $b_0= 17, c_0 = 4$.  If $b, c \equiv b_0, c_0 \imod{3 \cdot 5 \cdot 11 \cdot 29\cdot 47}$, $E_{b,c}$ has maximal
mod $4$ image, and $E_{b, c}$ is semi-stable, then $E_{b, c}$ has maximal $\ell$-adic image for all $\ell$.
\end{prop}
\begin{proof}
For a prime $\frakp$, let $E_\frakp$ denote $E_{b_0, c_0} \imod{\frakp}$. First, we use the prime $\frakp = (\alpha^2 + \alpha + 2)$
and compute that $\#E_\frakp(K/\frakp) = 2^3$.  Therefore $E_{b_0, c_0}$, the mod $\ell$ representation is surjective
for $\ell \neq 2, 31$.  Since $\frakp \cap \Z = (3)$, this shows that if $b, c \equiv b_0, c_0 \imod{3}$, then the same holds for $E_{b, c}$.

To show that the mod 8 representation is maximal for $E_{b_0, c_0}$, we consider the prime $\frakp = (3 \alpha^2 + 2)$.
As $\frakp \cap \Z = 29$, the same is true for any curve with $b, c \equiv b_0, c_0 \imod{29}$.

To show that the mod 9 representation is surjective for $E_{b_0, c_0}$, we use the prime $\frakp = (2 \alpha^2 + \alpha + 4)$,
which lies above $47$.  If $b, c \equiv b_0, c_0 \imod{47}$, the mod 9 representation is surjective for $E_{b, c}$.

Lastly, we must show that the mod 31 representation is surjective for $E_{b_0, c_0}$.  We let $w_1, w_2, w_3$ be primes such
that $w_1 \mid (-5 \alpha^2 + \alpha - 3)$, $w_2 \mid (5)$, and $w_3 \mid (\alpha^2 + \alpha + 2)$.  Since
$w_1 \cap \Z = 11$, $w_2 \cap \Z = 5$, $w_3 \cap \Z = 3$, if $b, c \equiv b_0, c_0 \imod{3 \cdot 5 \cdot 11}$,
the mod 31 representation is surjective.

Putting this together, we get, just as in the previous section, that the image of the mod $\ell$ representation is maximal for all $\ell$.
\end{proof}


\subsection{Proof of Theorem \ref{infinite}}
We have already developed congruence conditions for each condition.  We gather them together here.  For convenience, we use the notation
``$\textrm{cyc}_{p}$" to refer to the mod $p$ congruence condition in Proposition \ref{family_intersect}.
\begin{align*}
\begin{array}{c|c|c}
\text{Condition} & b & c\\\hline
\text{semistable} & 5 \imod{12} & 4 \imod{12}\\
\text{mod 9}      & 17\imod{47} & 4 \imod{47}\\
\text{mod 4}      &3699 \imod{7 \cdot 13 \cdot 31} & 4183 \imod{7 \cdot 13 \cdot 31}\\
\text{mod 31}     & 17 \imod{3 \cdot 5 \cdot 11}   & 4 \imod{3 \cdot 5 \cdot 11}\\
\text{mod 8}      & 17 \imod{29}& 4 \imod{29}\\
\text{$\ell$-adic}&2 \imod{3}   & 1 \imod{3}\\
\textrm{cyc}_{3}  & 2 \imod{3}  & 1 \imod{3}\\
\textrm{cyc}_{5}    & 2 \imod{5}  & 4 \imod{5}\\
\textrm{cyc}_{11}   & 6 \imod{11} & 4 \imod{11}\\
\textrm{cyc}_{17}   & 0 \imod{17} & 0 \imod{17}\\
\textrm{cyc}_{31}   & 10 \imod{31}& 29 \imod{31}\\
\textrm{cyc}_{787}  & 0 \imod{787}& 0\imod{787}\\
\textrm{cyc}_{827}  & 0 \imod{827}& 0\imod{827}
\end{array}
\end{align*}
Using the Chinese remainder theorem, we put these together to get the congruence condition from the statement of the theorem.
This completes the proof of Theorem \ref{infinite}.

\section{Non-Semistable Elliptic Curves and Nontrivial Narrow Class Group}\label{nonsemi}
Here we show how to extend Greicius's test for the surjectivity of the mod $\ell$ representation to arbitrary elliptic curves $E$ with non-integral $j_E$-invariant over arbitrary cubic number fields $K$. In particular, we do not need to suppose that the elliptic curve $E$ is semistable over $K$.

To fix notation, we let $\Sigma_K$ denote the set of primes of $K$ for a number field $K$. We let $S_E \subseteq \Sigma_K$ be
the set of primes of bad reduction of the elliptic curve $E/K$, and we let $S_\ell$ be the set of primes above the rational
prime $\ell$. For a prime $v \in \Sigma_K$, we let $I_v \subseteq G_K$ denote the inertia subgroup at $v$. If $w$ is a prime
lying over $v$, we denote by $f(w/v)$ and $e(w/v)$, the residue degree and ramification index of $w$ over $v$, respectively. If a Galois
extension is unramifed at $v$, we denote by $\Frob_v$ the Frobenius element of $v$ in the Galois group of that extension. We let
$U_v$ denote the $v$-adic units, $\Ov$ the completion of $\Ok$ at $v$, and $\frakm_v$ the maximal ideal of $\Ov$. We let $U_K$
denote the group of units of $\Ok$, and we let $U_K^+$ denote the subgroup of totally positive units. For a modulus $\frakm$,
in the sense of class field theory, we let $U_{\frakm,1}=U_K \cap K_{\frakm,1}$. All of this notation is identical to that of \cite{greicius} where applicable.

By a \emph{half-Borel} subgroup of $\GL_2(\Fl)$ 
we mean a subgroup stabilizing a subspace of $\Fl^2$ and acting
trivially on either the subspace or the quotient by it. In other words, some conjugate of the subgroup lies in the set of matrices of the form
$$\pmat{\ast}{\ast}{0}{1} \quad \text{or} \quad \pmat{1}{\ast}{0}{\ast}.$$
\begin{rem}
Throughout this section, words such as ``larger," ``smaller," and ``bounded" are used in the sense of divisibility rather than absolute value.
\end{rem}
To begin, we recall some useful results from \cite{greicius}. In \cite[Lemma 3.4]{greicius}, Greicius summarizes some results from \cite{serre}
on the inertia representation for the case of semistable elliptic curves. We restate these results, but
note that most of the results only depend on the curve being semistable at a particular prime.

\begin{lemma}[\cite{serre}]\label{serrelemma}
Let $K$ be a number field, $\ell$ a rational prime, and $E/K$ an elliptic curve with $j$-invariant $j_E$.
Fix $v \in \Sigma_K$ and $w \in \Sigma_{\cl{K}}$ with $w \mid v$.

First, suppose $v \notin S_E$.
\begin{enumerate}[(i)]
\item \label{good} If $v \notin S_\ell$, then $\rho_{E,\ell}(I_w)$ is trivial.
\end{enumerate}
Suppose for the next two that $\ell$ is unramified in $K$.
\begin{enumerate}[(i)]\setcounter{enumi}{1}
\item\label{ordinary} If $v \in S_\ell$, and $E$ has ordinary reduction at $v$, then the image of $D_w$ is contained in a
Borel subgroup, and $I_w$ acts via the trivial character and via a character of order $p-1$.
\item\label{supersingular} If $v \in S_\ell$, and $E$ has supersingular reduction at $v$, then the image of $I_w$ is a non-split Cartan subgroup.
\end{enumerate}
Suppose, furthermore, that $E$ is semistable at $v$.
\begin{enumerate}[(i)]\setcounter{enumi}{3}
\item \label{badnotdivides} If $v \in S_E \setminus S_\ell$, then $\rho_{E,\ell}(I_w)$ is trivial or cyclic of order $\ell$.
\item\label{tatepart} If $v \in S_E$, $\ell \nmid v(j_E)$, then $\rho_{E,\ell}(I_w)$ contains an element of order $\ell$.
\item If $v \in S_\ell$ and $\ell$ is unramified in $K$, then $\rho_{E,\ell}(I_w)$ is a half Borel subgroup.
\end{enumerate}
\end{lemma}

\begin{proof}
Assume first that $\ell$ is not ramified in $K$. Everything except (\ref{ordinary}) and (\ref{supersingular}) are stated in \cite[Lemma 3.4]{greicius}.
The remaining two follow from the content of \cite[Section 1.11]{serre}.

If $\ell$ is ramified in $K$, (\ref{good}) follows from the Criterion of N\'eron-Ogg-Shafarevich, and (\ref{badnotdivides}) follows by
examining the Tate curve as in \cite[IV.A.1.5]{serre98}.
\end{proof}

As in \cite{greicius}, we have the following:

\begin{prop}\label{isborel}
Suppose $E/K$ is an elliptic curve with $j$-invariant $j_E$, and suppose $\ell$ is a prime that does not ramify in $K$ such that
$\ell \nmid v(j_E)$ for some place $v$ of multiplicative reduction. If $\rho_{E,\ell}(G_K) \neq \GL_2(\Fl)$, then $\rho_{E,\ell}(G_K)$
is contained in a Borel subgroup of $\GL_2(\Fl)$.
\end{prop}

\begin{proof}
By Lemma \ref{serrelemma}(\ref{tatepart}), we know that the mod $\ell$ image contains an element of order $\ell$.
It thus must either contain $\SL_2(\Fl)$ or be contained in a Borel subgroup (see \cite{serre98}, IV.3.2, Lemma 2). As $\ell$ does not
ramify in $K$, the intersection of $K$ with $\Qb(\mu_{\ell^\infty})$ is $\Qb$, so the determinant
modulo $\ell$ is surjective. Therefore, in the former case, the image is $\GL_2(\Fl)$.
\end{proof}

\begin{rem}\label{extendisborel}
In practice, we find elliptic curves with one $v$ such that $v(j_E)=-1$, so that $v$ works for all $\ell$. However, one can extend the result
to other elliptic curves, if necessary by checking explicitly other $\ell$ using Proposition \ref{serreprop}.
Also note that if $\ell \neq 2,3$ (or just $\ell \neq 2$ if $E$ has full $2$-torsion over $K$), we only need $v$ to be potentially
multiplicative, since, as we will see below, $E$ becomes semistable over an extension of degree divisible only by the primes $2$ and $3$.
\end{rem}

Assuming $\ell \nmid v(j_E)$ for $v$ multiplicative, let us suppose that the mod $\ell$ representation is not surjective. It follows that the representation $\rho_{E,\ell}$
takes the form $$\pmat{\chi_{1}}{\ast}{0}{\chi_{2}}$$ in some basis of $E[\ell]$, where $\chi_1,\chi_2$ are characters from $G_K$
into $\Fl^\times$. Note that the Borel subgroup and hence the characters depend purely on the $1$-dimensional vector subspace of
$E[\ell]$ left invariant by that subgroup. We call these the \emph{isogeny characters relative to that subspace}.

In fact, the pair of isogeny characters is independent of which subspace we choose. Although we will not truly need this fact until Lemma \ref{dividesell}, we prove it now.

\begin{lemma}\label{isogind}The set of isogeny characters $\{\chi_1,\chi_2\}$ is independent of the basis for our Borel subgroup.\end{lemma}

\begin{proof}We can view $E[\ell]$ as a module over the ring $R:=\Fl[G_K]$ via the representation. If the image is contained in a Borel
subgroup, then it is reducible, i.e. there is some subspace
$W \subseteq E[\ell]$ which is an $R$-submodule of $E[\ell]$. The $R$-modules $W$ and $E[\ell]/W$ correspond to the characters
$\chi_i$. But if $W'$ is a different submodule, then the set of $R$-modules $\{W,E[\ell]/W\}$ equals the set $\{ W',E[\ell]/W' \}$
by the Jordan-H\"{o}lder theorem. It follows that the characters, and in particular, whether they ramify at a given prime of $K$, are independent of $W$.\end{proof}

In \cite{greicius}, Greicius shows that in the situation above, if $E$ is semistable, and $K$ satisfies certain technical
properties, then one of these isogeny characters
is unramified. In that case, because the narrow class group of $K$ is trivial, that character is trivial. Greicius uses this to show that $\ell$
divides the number of points modulo any prime of good reduction. 

As stated earlier, we would like to extend his results ruling out all but finitely many $\ell$ to non-semistable curves and more general $K$.
In order to do this, we would like to analyze how much the two characters $\chi_1,\chi_2$ can ramify.

First, we briefly note what happens at semistable primes not dividing $\ell$.

\begin{lemma}\label{anybasis}
Suppose that $\rho_\ell\colon G_K \to \GL(E[\ell])$ is contained in a Borel subgroup. Then the isogeny characters
are unramified at any $v \notin S_\ell$ at which $E$ has semistable reduction.
\end{lemma}

\begin{proof}
If $E$ has good reduction at $v$, then $\rho_\ell$ itself is unramified. If $E$ has multiplicative reduction
at $v$, we know by Lemma \ref{serrelemma} that the image of $I_v$ has order dividing $\ell$. As the characters
have codomain $F_\ell^\times$, which has order $\ell-1$, the characters are trivial on $I_v$.
\end{proof}

\subsection{Non-semistability}
Next, we would like to tackle the primes at which $E$ is not semistable. We will only cover those primes not dividing
$\ell$, then assume at the end that $\ell$ is a rational prime not lying under any prime of additive reduction.

We say that an elliptic curve $E/K$ is \emph{Legendre} over $K$ if $E$ is isomorphic over $K$ to a curve in Legendre form $y^2=x(x-1)(x-\lambda)$.
If the elliptic curve $y^2=f(x)$, where $f$ is a cubic polynomial over $K$, has full $2$-torsion over $K$, then we can write the equation
in the form $$y^2=(x-a)(x-b)(x-c)$$ with $a,b,c \in K$. Over $\cl{K}$, the curve can be written in Legendre form, but that does not
mean that the curve is Legendre over $K$. More precisely, we have the following well known fact (see Proposition III.1.7 in \cite{silverman1}).

\begin{lemma} If $E/K$ is an elliptic curve defined by equation $y^2=(x-e_1)(x-e_2)(x-e_3)$ over $K$ (which happens if $E$ has full
$2$-torsion over $K$), then $E$ is Legendre over $L=K(\sqrt{e_1-e_2})$. (By symmetry, this means we can take $L=K(\sqrt{e_i-e_j})$ for some $i \neq j$.)
\end{lemma}

In order to deal with non-semistability, we would like to know over what field the curve becomes semistable. An explicit form of the
semistable reduction theorem states
that this is true for an abelian variety if we adjoin the $12$-torsion (see for example, \cite[Proposition 5.10]{abbes}).
However, in the case of elliptic curves, we have the following lemma

\begin{lemma}\label{becomessemistable}
Suppose $E/K$ is an elliptic curve with full $2$-torsion over $K$ and additive reduction at $v \in \Sigma_K$. Then:
\begin{itemize}
\item If $E$ has potentially good reduction at $v$, then there is an extension $L/K$ of ramification index at most
$2$ over which $E$ has good reduction.
\item If $E$ has potentially multiplicative reduction at $v$, then there is an extension $L/K$ of ramification index
at most $4$ over which $E$ has multiplicative reduction.
\end{itemize}
Furthermore, suppose $E$ is Legendre over $L/K$. If $v$ does not ramify in $L$, then the above indices are $1$ and $2$, respectively.
\end{lemma}

\begin{proof}
By the proof of Proposition 5.4 in \cite[Chapter VII]{silverman1}, a Legendre curve has good reduction at all
primes of potentially good reduction, and otherwise has multiplicative reduction at $v$ after adjoining the square root of
a uniformizer at $v$. Finally, a curve with full $2$-torsion becomes Legendre over an at most quadratic extension of $K$.
\end{proof}


We know that over some extension $L/K$, one of the characters is trivial. Therefore,
we could count the points over the residue fields of $L$ and take the greatest common divisor of these values. We should note first that
computing in number fields is ineffective and impractical. Therefore, it makes more sense to compute $[L:K]$ and then take the unique finite
field of that degree (over the residue field of $K$), then compute the number of points of our elliptic curve over that finite field.

In order to trivialize one of the characters, one could take a field over which $E$ is semistable. Over this extension and under
the conditions above, one of our characters would be unramified at every finite prime not dividing $\ell$. It would follow that
that character becomes trivial over the narrow Hilbert class field of this extension. However, the analysis of primes dividing
$\ell$ does not work over this field, which means that we must find a different method.

We put bounds on the ramification of non-semistable primes at each of these characters. 
Then, since the extensions defined by the characters are abelian, they have conductors by class field theory, and
we can use these bounds on the ramification indices to control the conductor of this extension. Therefore, assuming one of the
characters is unramified at all primes of semistable reduction, the character is trivialized over the ray class field associated to this conductor.

One advantage is that the splitting of primes in the ray class field is described explicitly by the Artin reciprocity law,
so we can actually bound the order of the image of the Frobenius element under the character. Note that \textsc{Magma} can
compute the order of a given ideal in the ray class group, so this method can be easily implemented. Finally, note that once
we fix the field $K$, the computation of the ray class group for different moduli $\frakm$ is very easy by the exact sequence
$$1 \to U^+_K/U_{\frakm,1} \to (\Ok/\frakm)^\times \to \raym \to \Cc^\infty_K \to 1,$$ the only difficult part being
computing the narrow class group. Conversely, the method outlined above involves computing an ideal class group that
depends on the modulus and hence on the elliptic curve. Therefore, if one is working with many elliptic curves over
the same field, our method works much better.

We suppose for simplicity from now on that $2$ does not ramify in $K$.
We also suppose for the next two sections that the mod $\ell$ representation is contained in some (fixed) Borel subgroup,
hence has isogeny characters $\chi_i$ for $i=1,2$.

\subsection{Curves with full 2-torsion}

\begin{lemma}\label{ramificationbound}
Let $E/K$ have full $2$-torsion over $K$. Suppose $v \notin S_\ell$. Then the following are true

\begin{enumerate}

\item If $v \notin S_2$, each character $\chi_i$ is
not wildly ramified at $v$, and if $v \in S_2$, each character is ramified of index dividing $4$.

\item If $v \in S_2$ does not ramify in an extension over which $E$ is Legendre, or is of potentially good reduction, the index is $2$ (or $1$ if both happen).

\end{enumerate}

\end{lemma}

\begin{proof}
By Lemma \ref{anybasis}, we know that if $E$ has semistable reduction at $v$, then both characters are unramified at $v$.

By Lemma \ref{becomessemistable}, the curve becomes semistable at $v$ over an extension of degree $4$ obtained by first putting the
curve in Legendre form (which happens over some quadratic extension) and then adjoining the square root of a uniformizer at $v$.
If it is potentially good at $v$, this happens after obtaining Legendre form. Over this extension, $E$ is semistable at $v$, so
both of the isogeny characters are unramified at $v$ by Lemma \ref{anybasis}. If $v \nmid 2$, this means that the extension is not
wildly ramified at $v$. If $v \mid 2$, the above follows immediately from the description of the field over which $E$ becomes semistable at $v$ in Lemma \ref{becomessemistable}.
\end{proof}

We recall that for a place $v$ and an abelian extension $L/K$, there is a ``conductor" of the extension at $v$, which is
the smallest power $\frakm_v^k$ of $\frakm_v$, the maximal ideal at $v$, such that $1+\frakm_v^k$ is in the kernel of the
local Artin map. The conductor of a character $\chi_i$ is the conductor of the extension defined by $\chi_i$. The
\emph{global conductor} is the product of the local conductors. An abelian extension of given global conductor is contained in the
ray class field associated to that conductor, and a character with that global conductor is trivial over that ray class field.
We say that the conductor of $L/K$ \emph{outside a set $S \subseteq \Sigma_K$} is the product of the local conductors at all $v \notin S$.

\begin{prop}\label{modtwoconductor}
Suppose $E$ is Legendre over $L/K$. For $v \in S_2$, let $r(v)=3$ if $v$ ramifies in $L$ or $E$ has potentially
multiplicative reduction at $v$, $0$ otherwise. Let $s(v)=1$ if $f(v/2)=1$, $v$ ramifies in $L$, and $E$ has potentially
multiplicative reduction at $v$, and $0$ otherwise. Then let
\begin{align*}
i(v) := \begin{cases}
0 & \text{if $E$ is semistable at $v$},\\
0 & \text{if $E$ has additive potentially good reduction at $v$,}\\
 &  \text{\quad $v$ does not ramify in $L$, and $v \nmid 2$,}\\
1 & \text{if $E$ has additive reduction at $v$ and $v \nmid 2$ otherwise.}\\
r(v)+s(v) &\text{otherwise}.\end{cases}
\end{align*}
Then if we let $\frakm_f = \prod_{v \in \Sigma_K} v^{i(v)}$ and $\frakm$ the product of $\frakm_f$ with all real places, the conductor of the extension defined by
$\chi_i$ is at most $\frakm$ outside $S_\ell$.
\end{prop}
\begin{proof}
Lemma \ref{ramificationbound} gives us bounds on the ramification of the extension defined by the characters. We need only
translate these into bounds on the conductor of the extension.

As mentioned, the extension is unramified at any primes of semistable reduction. We also know the extension is unramified at $v$ if $v$ does not ramify
in $L$ and is potentially good. Furthermore, as the index is a power of $2$, it is tamely ramified at all primes other than
$2$. This means that the conductor has exponent at most $1$ at that prime.

In the case of $v\mid 2$, we use the description of the local Artin map to compute the conductor. The image of $U_v$ under
the local Artin map is the inertia group at $v$ in the extension, and we know that this image is of order bounded by $4$ (respectively,
$2$), so its kernel has index at most $4$ (respectively, $2$). We thus wish to classify subgroups of $U_v$ of index at most $4$ or $2$.
Since $2$ is not ramified in $K$, the $v$-adic exponential converges for $|x|_v < 1/2$, meaning that $1+\frakm_v^2 \Ov$,  is
isomorphic to $\Ov$ as a topological group. This latter is pro-cyclic, hence has a unique open subgroup of given index. If the index
is at most $4$, and $f(v/2)>1$, then $(1+\frakm_v^2 \Ov)/(1+\frakm_v^3 \Ov)$ has order at least $4$, so this kernel contains $1+\frakm_v^3 \Ov$,
and the conductor is at most $\frakm_v^3$ at $v$.

In the case that the index is only bounded by $4$ (by the previous proposition), and $f(v/2)=1$, then $(1+\frakm_v^2 \Ov)/(1+\frakm_v^4 \Ov)$ has order $4$,
so the conductor is bounded by $\frakm_v^4$ at $v$.
\end{proof}

\begin{rem}
If one does not wish to compute $L$ above, one can always suppose that $v$ ramifies in $L$. Then $r(v)$ is always $3$, and $s(v)=1$ if
$v(j)<0$ and $f(v/2)=1$, and $s(v)=0$ otherwise. This is what we did in our calculations. However, we included the most general result
for completeness, and because it may reduce the computation time for much larger fields.
\end{rem}

\subsection{General curves}
We now do not assume the curve has full $2$-torsion in order to apply our results to the search for curves whose adelic representation is surjective. Once again, we suppose from now on for simplicity that $3$ does not ramify in $K$.

We know that after passing to a degree $6$ extension, the curve has full $2$-torsion. It follows that, after passing to an extension of
degree at most $24$, our curve becomes semistable. We therefore get the same bounds for ramification indices as in Lemma \ref{ramificationbound},
but we must multiply each of them by $6$. We translate these into bounds on the conductor of the character:

\begin{prop}\label{generalconductor}
For $v \in S_2$, let $r(v)=2$ if $E$ has potentially good reduction at $v$ and $r(v)=3$ if $E$ has potentially multiplicative reduction at $v$.
\begin{align*}
i(v) := \begin{cases}
0 & \text{if $E$ is semistable at $v$,}\\
1 & \text{if $E$ has additive reduction at $v$ and $v \nmid 6$,}\\
2 &\text{if $E$ has additive reduction at $v$ and $v \in S_3$,}\\
2+\left\lceil \frac{r(v)}{f(v/2)} \right\rceil & \text{if $E$ has additive reduction at $v$ and $v \in S_2$.}\\
\end{cases}
\end{align*}
Then if we let $\frakm_f = \prod_{v \in \Sigma_K} v^{i(v)}$ and $\frakm$ the product of $\frakm_f$ with all real primes, then the conductor of the extension defined by
$\chi_i$ is at most $\frakm$ outside $S_\ell$.
\end{prop}
\begin{proof}
Once again, if $E$ is semistable at $v$, there is no ramification. If $v \nmid 6$, then the ramification is not wild at $v$,
so the conductor has power at most $1$ at $v$.

As for $v\mid 3$, the $v$-adic logarithm shows that $1+\frakm_v \Ov \cong \Ov$ as topological groups (assuming that $3$ does
not ramify in $K$), so the kernel of the character must contain $1+\frakm_v \Ov$.

Finally, if $v \mid 2$, then (assuming $2$ does not ramify in $K$), the group $1+\frakm_v^2 \Ov$ is pro-cyclic, and
$(1+\frakm_v^k \Ov)/(1+\frakm_v^{k+1} \Ov)$ has order $f(v/2)$, so the exponent of the conductor is at most
$2 + \left\lceil r(v)/f(v/2)\right\rceil$ if the $2$-part of the ramification index is at most $2^{r(v)}$.
The values of $r(v)$ follow from the values from the previous section, but with ramification indices multiplied by $6$.
\end{proof}

\subsection{Cubic Fields}
Next, we deal with $v \in S_\ell$. In \cite{greicius}, Greicius uses Corollary 3.6 to analyze this case. This relies on
the fact that $K = \Q(\alpha)$, $\alpha^{3} + \alpha +  1 = 0$, $\alpha \in \R$ has trivial narrow class group and the
fact that there exists a totally positive unit $u$ of $K$ such that $u - 1$ is also a unit.
To extend his result, we prove the following proposition for general cubic fields:

\begin{lemma}\label{dividesell}
Let $E/K$ be an elliptic curve over a cubic field $K$, and let $\ell$ satisfy the conditions
of Proposition \ref{isborel} and not lie below any primes of additive reduction for $E$. Let $u$ be a positive
fundamental unit of $K$. Let $k$ be the smallest positive power of $u$ contained in $U_{\frakm,1}$. Let $d:=\#\Cc^\infty_K$
denote the order of the narrow class group of $K$. Let $r:=\# (\Ok/\frakm_f)^\times,$ where $\frakm$ is defined as above,
depending on whether $E$ has full $2$-torsion over $K$. Suppose that $$\ell \nmid \N_{K/\Qb}\left(\prod_{i=1}^{dr/k} u^{ik}-1\right).$$
Then one of the characters $\chi_i$ is unramified at every semistable $v \in S_\ell$.
\end{lemma}
\begin{proof} We divide our argument into cases based on the splitting of $\ell$ in $K$.

Case 1: \emph{$\ell$ is inert}. Then there is a unique place $v$ of $K$ above $\ell$. As the image of $\rho_\ell$ contains an
element of order $\ell$, the image cannot be contained in a non-split Cartan subgroup, so $E$ has ordinary or bad reduction
at $v$ by Lemma \ref{serrelemma}(\ref{supersingular}). In
each case, Lemma \ref{serrelemma} tells us that the image of inertia is a half-split Cartan subgroup or a half-Borel subgroup,
implying that the character is unramified at $v$.

Case 2: \emph{$\ell$ splits completely}. Suppose $(\ell)=\frakp \frakq \frakr$. By the previous case, we see that each prime
above $\ell$ is unramified at at least one of the characters. By the pigeonhole principle, one of the characters, say $\chi_i$,
is unramified at at least two of the primes above $\ell$. Suppose without loss of generality that it is unramified at $\frakq$
and $\frakr$, and suppose it is ramified at $\frakp$.

As $\# \Fl^\times  = \ell-1$, which is prime to $\frakp$, the character $\chi_i$ is not wildly ramified at $\frakp$. Its
conductor is therefore at most $\frakm \frakp$, by either Proposition \ref{modtwoconductor} or Proposition \ref{generalconductor}, respectively. It follows that $\chi_i$
factors through the ray class field $L_{\frakm \frakp}$, i.e. it factors through $\Cc^{\frakm \frakp}_K$ under the Artin map.
Now, we have an exact sequence
\[1 \to U^+_K/U_{\frakm \frakp,1} \to (\Ok/\frakm_f \frakp )^\times  \to \Cc^{\frakm \frakp}_K \to \Cc^\infty_K \to 1.\]

Suppose that $\chi_i$ is ramified at $\frakp$. By the hypothesis, $u^{ik}-1$ is prime to $\ell$, hence not divisible by $\frakm_f \frakp$, for all $ik \le dr$. For
$i$ such that $k \nmid i \le dr$, we know that $u^i-1$ is not divisible by $\frakm_f \frakp$. Therefore, we know that
$u$ has order greater than $dr$ in $U^+_K/U_{\frakm \frakp,1}$, and the same is true of its image in $(\Ok/\frakm \frakp )^\times$.
Now, $(\Ok/\frakm_f \frakp )^\times \cong (\Ok/\frakm_f)^\times \times (\Ok/\frakp)^\times$ as $\ell$ does not lie below any primes of
additive reduction, hence any primes dividing $\frakm$. This means that $$\#(\Ok/\frakm_f \frakp )^\times = \ell r.$$ As $u$ has
order greater than $dr$, the cokernel of $$U^+_K/U_{\frakm \frakp,1} \to (\Ok/\frakm_f \frakp )^\times$$ has order less than $\ell/d$.

Let $\Dc$ denote this cokernel. By our exact sequence above, $\Dc$ is also equal to the kernel of the surjection $\Cc^{\frakm \frakp}_K \to \Cc^\infty_K$.
Then $\Dc$ has index $d=\# \Cc^\infty_K$ in
$\Cc^{\frakm \frakp}_K$. Since $\chi_i \colon \Cc^{\frakm \frakp}_K \to \Fl^\times$ is surjective, the image of $\Dc$ in $\Fl^\times$
has index at most $d$, hence is of order at least $\ell/d$. But we have already shown that $\Dc$ has order less than $\ell/d$,
so it cannot surject onto a group of order at least $\ell/d$. This contradiction shows that $\chi_i$ cannot be ramified at $\frakp$.

Case 3: \emph{$\ell$ splits as $\frakp \frakq$ where $f(\frakp/\ell)=1$}. Then one of the characters is unramified at
$\frakq$, so we apply the same argument with that character at $\frakp$ to show that it cannot be ramified at $\frakp$.

This completes the proof.
\end{proof}

\begin{rem}If $K=\Qb$, the lemma is automatically true because only Case 1 can occur.\end{rem}

We now combine our results into the following.

\begin{thm}\label{nonsemi_thm}
Let $E/K$ be an elliptic curve over a cubic field $K$ not Galois over $\Qb$. If $E$ has full $2$-torsion over $K$,
let $\frakm$ be as in Proposition \ref{modtwoconductor}, and otherwise, let $\frakm$ be as in Proposition
\ref{generalconductor}. Let $\ell$ be a prime of $\Qb$ such that:
\begin{enumerate}
\item $\ell$ does not ramify in $K$
\item \label{jorder}There is a prime $v$ at which $E$ has semistable reduction and such that $\ell \nmid v(j_E)<0$
\item $\ell$ does not lie below any primes of $K$ at which $E$ is not semistable
\item We have $$\ell \nmid \N_{K/\Qb}\left(\prod_{i=1}^{dr/k} u^i-1\right),$$ where $d:=\#\Cc^\infty_K$, $r:=\#(\Ok/\frakm_f)^\times$,
$u$ is a positive fundamental unit of $K$, and $k$ is the order of $u$ in $U^+_K/U_{\frakm,1}$
\end{enumerate}

If $\rho_\ell(G_K) \subsetneq \GL_2(\Fl)$, then one of the characters $\chi_i$ is trivial over the ray class field $L_\frakm$ of $K$ of conductor $\frakm$.
\end{thm}

\begin{proof}
By Proposition \ref{isborel}, the first two conditions on $\ell$ imply that the characters exist. Proposition \ref{modtwoconductor} or
Proposition \ref{generalconductor}, respectively, shows that character has conductor $\frakm$ outside of $S_\ell$. The third condition
means that $E$ is semistable at all of $S_\ell$, so Lemma \ref{dividesell} implies that one of the characters $\chi_i$ is unramified at
all of $S_\ell$.  As it is unramified at $S_\ell$, we know that $\frakm$ is its conductor, and thus it is trivial over $L_\frakm$.
\end{proof}

\begin{cor}\label{cardtest_cor}
Let $\mf{m}$, $E/K$, and $\ell$ be as above, and let $n$ be the order of $v$ in $\# \raym$, where $v \nmid \frakm$ is a prime of good reduction.
Let $l_v$ be the unique degree $n$ extension of $k_v$. Then if $\rho_{E,\ell}(G_K) \neq \GL_2(\Fl)$, we have $\ell \mid \#\wt{E}_{v}(l_{v})$ (equivalently, the reduced curve has an $\ell$-torsion point).
\end{cor}

\begin{proof}
This argument follows \cite[Proposition 3.8]{greicius} almost identically.
\end{proof}

\begin{rem}
Note that this is also true if we replace $n$ by one of its multiples, since then we only get a larger $l_v$. In particular,
if we do not wish to compute the order of $v$ in the ray class group, we can just let $n$ always equal $\# \raym$.
\end{rem}

This gives us an efficient way to test for surjective image. In searching for examples, we restricted our attention to those $E$ with multiplicative reduction at a prime $v$ such that $v(j_E)=-1$. For the other $\ell$, we use Proposition \ref{serreprop}.


\begin{rem}\label{critremark}We can replace (\ref{jorder}) in each criterion with $v$ potentially multiplicative, as
explained in Remark \ref{extendisborel}. Therefore, in most generality, our method extends to arbitrary elliptic curves with non-integral $j$-invariant over arbitrary cubic (and quadratic) fields $K$.\end{rem}


\subsection{Further examples}
In this section we briefly mention two examples that result from our theory developed above.
These examples were found using Theorem \ref{nonsemi_thm} and Corollary \ref{cardtest_cor} and methods similar
to Section 3.4 of \cite{greicius} and Example \ref{single_example}.
Due to the similarities, we content ourselves with just briefly mentioning the proof.

We first give an example of a non-Galois cubic field $K$ and an elliptic curve $E/K$ with full $2$-torsion over $K$ which has maximal
adelic Galois image, but is not semistable.
\subsubsection{Proof of Example \ref{nonsemi_ex}}
Let $\beta$ be the unique real root of $x^{3} + 4x^{2} + 7x - 4$ and $K = \Q(\beta)$.
We want to show that the adelic representation $\rho_{E}$ associated to
the elliptic curve $$E/K \colon y^{2} = x(x + (10\beta^{2} - 3))(x - (\beta + 4))$$
surjects onto $V_1(2)  \times \prod_{\ell > 2} \gl_2(\mathbb{Z}_{\ell})$.

Let $H := \rho_{E}(G_{K})$. Note that the discriminant of $K$ is $-503$, $5\beta^{2} + 9\beta - 5$ is a totally positive unit of $K$,
and $\mf{m}_{f} = ((3/2)\beta^{2} + (13/2)\beta + 13)^{3}$.
In condition $(4)$ of Theorem \ref{nonsemi_thm} we will be inefficient and take $k = 1$.
Conditions $(1), (3)$, and $(4)$ of Theorem \ref{nonsemi_thm} imply that $H(\ell) = \gl_{2}(\mathbb{F}_{\ell})$ for
$\ell \neq 2$, $5$, $17$, $41$, $73$, $211$, $503$, $2143$, $2269$, $3907$, $5449$, $31741$, $40471$, $493333$,
$938251$, $1225603$, $1315849$, $37012153$. To show that $H(\ell) = \gl_{2}(\mathbb{F}_{\ell})$ for these remaining $\ell \neq 2$,
we use Proposition \ref{serreprop}. Then proceeding as in Example \ref{single_example} shows that $\rho_{E}$ surjects onto
$V_1(2)  \times \prod_{\ell > 2} \gl_2(\mathbb{Z}_{\ell})$.
(See \texttt{nonsemi\_magma.txt} and \texttt{nonsemi\_sage.sws} in \cite{transcript} for results of computations.)\hfill\qed

We end by noting that the methods we developed in Section \ref{nonsemi} can also be used to consider the case
of surjective $\rho_{E}$ when the narrow class group is nontrivial. We offer the following example.
\subsubsection{Proof of Example \ref{nontriv_ex}}
Let $\beta$ be the unique real root of $x^{3} + 8x^{2} - 3x + 1$ and $K = \Q(\beta)$. We show that with
$$E/K \colon y^{2} + xy + \beta y = x^{3} - 8x^{2} - 6x - 1$$
the corresponding $\rho_{E}$ surjects onto $\gl_{2}(\widehat{\mathbb{Z}})$.

Again as above, let $H := \rho_{E}(G_{K})$.
We note that $\Q(\beta)$ has discriminant $-1823$ and narrow class group being $C_{2}$.
Also note that $-\beta$ is a totally positive unit and $\nm((- \beta - 1)(\beta^{2} - 1)) = 7^{2} \cdot 11$.
The conditions given in Theorem \ref{nonsemi_thm} give that $H(\ell) = \gl_{2}(\F_{\ell})$ for $\ell \neq 7, 11, 1823$.
For these three $\ell$, we can use Proposition \ref{serreprop}. Then using similar methods as in Section 3.4 of \cite{greicius}
shows that the adelic Galois representation surjects onto $\gl_{2}(\widehat{\mathbb{Z}})$.
(See \texttt{nontriv\_narrow\_magma.txt} and \texttt{nontriv\_narrow\_sage.sws} in \cite{transcript} for results of computations.) \hfill\qed

\section{Remarks on higher torsion}
Although we initially raised the problem of explicitly computing Galois representations for elliptic curves with arbitrary specified torsion data, we have primarily worked with $2$-torsion in this paper.  After all, it is significantly easier to produce examples of elliptic curves with $2$-torsion, and to work with explicit formulas for the $2$-division-polynomials. This is advantageous, for example, in verifying that the ``mod $4$'' Galois representation is maximal (i.e. $[K(E[4]):K] = 16$), and in producing the appropriate congruence conditions in Section \ref{infinite}

However, many of our methods can be applied to elliptic curves with more general torsion, as we now demonstrate. We will be content to give a discussion of examples, rather than attempting to write out fully general results.

As we have observed, torsion data for $E/K$ amounts to giving some restrictions on the ``mod $\ell^k$'' representation
\[
\rho_{E,\ell^k}(G_K) \subset \GL_2(\mathbb{Z}/\ell^k \mathbb{Z}).
\]
We are interested in knowing when the adelic Galois representation is maximal given such a restriction. One has to take different steps to verify that the $\ell$-adic Galois representation is maximal, and also account for a different abelianization map. The latter is equivalent to computing the commutator subgroup of the desired $\ell$-adic image, which is reduced to a finite computation by the refinement lemmas.

Likewise, one verifies that the $\ell$-adic image is maximal by checking that the ``mod $\ell^k$" image is maximal and that the image of the Galois representation also contains the entire kernel of the reduction map $\rho_{E,\ell^{k+1}}(G_K)  \rightarrow \rho_{E, \ell^k}(G_K)$. The first matter depends on the particular nature of the torsion data given; for instance, we consider the case of full $\ell$-torsion over $K$, in which case there is nothing to check, but if there is just one cycle of $\ell$-torsion points then one is looking for a ``half-Borel'' subgroup. We will not discuss the various cases, but instead focus on the problem of passing from finite to infinite: assuming that $ \rho_{E, \ell^k}(G_K)$ is known, how may we check that $\rho_{E, \ell^{\infty}}$ is the pre-image of $ \rho_{E, \ell^k}(G_K)$ under the reduction mod $\ell^k$ map? In the paper, we simply used the four-division polynomials for $E$, but this is infeasible for large $\ell$, so we indicate a method by analyzing Frobenius elements. Note that the Chebotarev Density Theorem certainly implies that if the $\ell$-adic Galois representation is maximal as desired, then a positive density of such elements must exist.

We illustrate by considering a particular example. Suppose $E/K$ has exactly one $\ell$-torsion cycle for some prime $\ell > 2$. Suppose we want to show that $\rho_{E, \ell^{\infty}}(G_K)$ is the full pre-image of the half-Borel subgroup
\[
\begin{pmatrix} 1 & * \\ 0 & * \end{pmatrix} \subset \GL_2(\mathbb{Z}/\ell\mathbb{Z}).
\]
By Serre's refinement lemma, it suffices to show that this half-Borel is in fact the mod $\ell$ image, and to show that $\rho_{E, \ell^2}(G_{\mathbb{Q}})  \supset  I + 2 \text{Mat}_{2 \times 2}(\mathbb{Z}/\ell \mathbb{Z})$. To address the first question, we see that the mod $\ell$ image must be the half-Borel or the half-Cartan:
\[
\begin{pmatrix} 1 & * \\ 0 & * \end{pmatrix}  \hspace{1cm} \text{or} \hspace{1cm} \begin{pmatrix} 1 & 0 \\ 0 & * \end{pmatrix}
\]

Now, suppose $\sigma = \Frob_{\mf{p}}$ is found whose action on the $\ell$-torsion points has characteristic polynomial factoring as
\[
T^2-a_pT + p \equiv  (T-1)^2 \imod{\ell}.
\]
We would like to know whether or not $\sigma$ is contained in the half split-Cartan subgroup. Since $(\rho_{E,\ell}(\sigma)-I)^2 \equiv 0 \imod{\ell}$, if it is contained in the half split-Cartan then $\sigma = I + \ell M$ for some $M \in \GL_2(\mathbb{Z}_{\ell})$. The characteristic polynomial of $M$ is then
\[
T^2 - \frac{a_p-2}{\ell} T + \frac{1+p - a_p}{\ell^2}.
\]
In particular, $\ell^2 \mid 1+p-a_p$. So if this ever fails, we know that $\rho_{E,\ell}$ was not contained in the split Cartan.

We now address the second question, which is a generalization of the arguments in \cite{lang_trotter}, p. 56-57. This group $I + 2 \text{Mat}_{2 \times 2}(\mathbb{Z}/\ell \mathbb{Z})$ is an abelian group isomorphic to $(\mathbb{Z}/\ell \mathbb{Z})^4$; we must obtain four independent elements.

We recall the following fact: if the characteristic polynomial of a $2 \times 2$ matrix with coefficients in $\mathbb{Z}_{\ell}$ factors with \emph{distinct} roots when considered modulo $\ell$, then it is diagonalizable over $\mathbb{Z}_{\ell}$. Note that it is \emph{not} sufficient for the roots to be distinct modulo $\ell^2$ (indeed, the polynomial may very well have four roots in $\mathbb{Z}/\ell^2 \mathbb{Z}$).

\underline{Step One.} It is easy to obtain diagonalizable, non-scalar elements. As outlined in the paper, we may compute the characteristic polynomial of a Frobenius element at $\sigma = \Frob_{\mf{p}}$:
\[
T^2 - a_p T + p,
\]
where $a_p = p+ 1 - \# E(\mathbb{F}_p)$. Suppose this polynomial factors mod $\ell^{2}$ as a product of two linear polynomials with \emph{distinct} roots $\lambda_1, \lambda_2$ such that $\lambda_1 \not\equiv \lambda_2 \imod{\ell}$. Then $\rho_{E, \ell^{\infty}}(\Frob_{\mf{p}})$ may be diagonalized over $\mathbb{Z}_{\ell}$, so it will have the form
\[
\sigma \equiv \begin{pmatrix} \lambda_1 + \ell^2 x_1 & 0 \\ 0 & \lambda_2 + \ell^2 x_2 \end{pmatrix}.
\]
for some $x_1, x_2 \in \mathbb{Z}_{\ell}$. Raising to the $(\ell-1)^{\text{st}}$ power yields
\[
\sigma^{\ell-1} \equiv I  + \ell \begin{pmatrix} \frac{\lambda_1^{\ell-1}-1}{\ell} & 0 \\ 0 & \frac{\lambda_2^{\ell-1}-1}{\ell}  \end{pmatrix} \imod{\ell^2}.
\]
(The point is that $\lambda_1^{\ell-1} \equiv 1 \imod{\ell}$, so we may legitimately do this division). As long as $\lambda_1^{\ell-1} \not\equiv \lambda_2^{\ell-1} \imod{\ell^2}$, this matrix is guaranteed to be diagonal but not scalar.
\begin{rem}
This can only fail if the image of Galois mod $\ell$ is trivial, since this only fails to occur when the image is contained in the normalizer of a non-split Cartan. This is impossible under our torsion assumptions. On the other hand, the image mod $\ell$ is trivial only for small $\ell$, in particular $\ell = 2$ for $K= \mathbb{Q}$.
\end{rem}

\underline{Step Two.} Next, suppose that a Frobenius element $\sigma$ is found with characteristic polynomial as above, but such that $\lambda_1^{\ell-1} \equiv \lambda_2^{\ell-1} \imod{\ell^2}$. Then we are guaranteed that $\sigma^{\ell-1}$ \emph{is} a scalar matrix. Adopting the basis with respect to which the matrix from the previous part is diagonal, we see that in this basis we would have \emph{all} diagonal matrices. (The point is that we have no way of telling in which bases the diagonal matrices obtainable in (1) are diagonalizable. However, scalar matrices are scalar in \emph{every} basis! We cannot obtain scalar matrices by the method in (1) since their eigenvalues are obviously guaranteed not to be distinct modulo $\ell$.) 

\underline{Step Three.} Now, suppose that we have produced $\sigma \in \rho_{\ell^{\infty}}$ with $\sigma = I + \ell M$,  for some $M \in \GL_2(\mathbb{Z}_{\ell})$.  The characteristic polynomial of $M$ is then
\[
T^2 - \frac{a_p-2}{\ell} T + \frac{1+p - a_p}{\ell^2}.
\]
and if this is irreducible $\imod{\ell}$ then we know that $\sigma$ cannot be put into upper-triangular form in any basis. Adjusting by diagonal elements already obtained, we obtain an element of the form
\[
I + \ell \begin{pmatrix} 0 & b \\ c & 0 \end{pmatrix} \imod{\ell^2}
\]
where $b,c \neq 0$. 

It remains to discuss how to produce such a sigma. Suppose $\sigma = \Frob_{\mf{p}}$ is found with characteristic polynomial factoring as
\[
T^2-a_pT + p \equiv  (T-1)^2 \imod{\ell}.
\]
which implies that $(\sigma-I)^2 \equiv 0 \imod{\ell}$. If $\sigma$ is congruent to the identity modulo $\ell$, then we know that $\sigma = I+ \ell M$. We test this in the example below by checking for full $\ell$-torsion in $E(\mathbb{F}_{\ell})$. An alternative approach is to observe that $\sigma^{\ell}$ \emph{must} be congruent to the identity modulo $\ell$, and we can compute its characteristic polynomial from that of $\sigma$ (We have $\sigma^{\ell} = A \sigma + B$ by using the relation of the characteristic polynomial, and similarly obtain another linear relation for $\sigma^{2\ell}$). 

\underline{Step Four.} Finally, we observe that we may conjugate any element of $G$ by any other element. We have already shown that we have all diagonal matrices, and the previous part guarantees a matrix of the form $I + \ell \left(\begin{smallmatrix} 0 & b \\ c & 0 \end{smallmatrix}\right) \imod{\ell^2}$. Conjugating this by $\left(\begin{smallmatrix} 1 & 0 \\ 0 & u \end{smallmatrix}\right)$ where $u \neq 1$, we obtain
\[
\begin{pmatrix} 0 & u^{-1} b \\ uc & 0 \end{pmatrix}.
\]

\subsubsection{Proof of Example \ref{tor_ex}}
To conclude, we carry out the program described above to give an example of an elliptic curve $E/\Q$ and Frobenius elements satisfying the above
properties. Let $$E/\Q \colon y^{2} + xy + y = x^{3} - x^{2} - 19353x + 958713$$ which has torsion subgroup
$\Z/7\Z$. In this case, $\ell = 7$. We compute that Step 1 above is satisfied by considering the Frobenius corresponding to 61,
Step 2 is satisfied by considering the Frobenius corresponding to 971. We then note that the prime $q = 127$
is such that $T^{2} - a_{q}T + q \equiv (T - 1)^{2}\imod{7}$ and $49 \nmid 1 + q - a_{q}$. Moreover, the prime $p = 19993$
is such that $T^{2} - a_{p}T + p \equiv (T - 1)^{2}\imod{7}$ and $49 \mid 1 + p - a_{p}$, $T^{2} - (a_{p} - 2)/\ell T + (1 + p - a_{p})/\ell^{2}$
is irreducible mod $\ell$ and $E(\F_{19993})$ has full $7$-torsion. This implies that Step 3 is satisified.
(See \texttt{frob\_find.txt} in \cite{transcript} for computations.) \hfill\qed



\bibliographystyle{amsalpha}
\bibliography{ref}

\providecommand{\bysame}{\leavevmode\hbox to3em{\hrulefill}\thinspace}
\providecommand{\MR}{\relax\ifhmode\unskip\space\fi MR }
\providecommand{\MRhref}[2]{%
  \href{http://www.ams.org/mathscinet-getitem?mr=#1}{#2}
}
\providecommand{\href}[2]{#2}
\begin{thebibliography}{Zyw10b}

\bibitem[Abb00]{abbes}
Ahmed Abbes, \emph{R\'eduction semi-stable des courbes d'apr\`es {A}rtin,
  {D}eligne, {G}rothendieck, {M}umford, {S}aito, {W}inters, {$\ldots$}},
  Courbes semi-stables et groupe fondamental en g\'eom\'etrie alg\'ebrique
  ({L}uminy, 1998), Progr. Math., vol. 187, Birkh\"auser, Basel, 2000,
  pp.~59--110.

\bibitem[Ade01]{torsion}
Clemens Adelmann, \emph{The decomposition of primes in torsion point fields},
  Lecture Notes in Mathematics, vol. 1761, Springer-Verlag, Berlin, 2001.

\bibitem[CFLTL]{transcript}
David Corwin, Tony Feng, Zane~Kun Li, and Sarah Trebat-Leder, \emph{Transcript
  of computations}, Available at
  \mbox{\url{http://code.google.com/p/maximal-adelic-image/}}.

\bibitem[Duk97]{duke}
William Duke, \emph{Elliptic curves with no exceptional primes}, C. R. Acad.
  Sci. Paris S\'er. I Math. \textbf{325} (1997), no.~8, 813--818.

\bibitem[Gre10]{greicius}
Aaron Greicius, \emph{Elliptic curves with surjective adelic {G}alois
  representations}, Experiment. Math. \textbf{19} (2010), no.~4, 495--507.

\bibitem[Jon10]{jones}
Nathan Jones, \emph{Almost all elliptic curves are {S}erre curves}, Trans.
  Amer. Math. Soc. \textbf{362} (2010), no.~3, 1547--1570.

\bibitem[LT76]{lang_trotter}
Serge Lang and Hale Trotter, \emph{Frobenius distributions in {${\rm
  GL}\sb{2}$}-extensions}, Lecture Notes in Mathematics, Vol. 504,
  Springer-Verlag, Berlin, 1976, Distribution of Frobenius automorphisms in
  ${{\rm{G}}L}\sb{2}$-extensions of the rational numbers.

\bibitem[Ser72]{serre}
Jean-Pierre Serre, \emph{Propri\'et\'es galoisiennes des points d'ordre fini
  des courbes elliptiques}, Invent. Math. \textbf{15} (1972), no.~4, 259--331.

\bibitem[Ser98]{serre98}
\bysame, \emph{Abelian {$l$}-adic representations and elliptic curves},
  Research Notes in Mathematics, vol.~7, A K Peters Ltd., Wellesley, MA, 1998,
  With the collaboration of Willem Kuyk and John Labute, Revised reprint of the
  1968 original.

\bibitem[Sil09]{silverman1}
Joseph~H. Silverman, \emph{The arithmetic of elliptic curves}, second ed.,
  Graduate Texts in Mathematics, vol. 106, Springer, Dordrecht, 2009.

\bibitem[Zyw10a]{zyw1}
David Zywina, \emph{Elliptic curves with maximal {G}alois action on their
  torsion points}, Bull. Lond. Math. Soc. \textbf{42} (2010), no.~5, 811--826.

\bibitem[Zyw10b]{zyw2}
\bysame, \emph{Hilbert's irreducibility theorem and the larger sieve}, Preprint
  (2010), arXiv:1011.6465v1.

\end{thebibliography}

\end{document}